\documentclass[11pt]{amsart}

\usepackage[utf8]{inputenc}
\usepackage{amssymb,amsmath}
\usepackage[T1]{fontenc}
\usepackage[english]{babel}

\allowdisplaybreaks

\newtheorem{theorem}{Theorem}[section]
\newtheorem{proposition}[theorem]{Proposition}

\newtheorem{corollary}[theorem]{Corollary}
\theoremstyle{definition}

\numberwithin{equation}{section}

\DeclareMathOperator{\dist}{dist}
\DeclareMathOperator{\ad}{ad}
\DeclareMathOperator{\Hom}{Hom}

\DeclareMathOperator{\traza}{Tr}

\begin{document}

\title[Module homomorphisms between non-commutative $L^p$-spaces]
{Hyperreflexivity of the space of module homomorphisms between 
non-commutative $L^p$-spaces}

\author{J. Alaminos} 
\author{J. Extremera}
\author{M. L. C. Godoy} 
\author{A. R. Villena}
\address{Departamento de An\' alisis
Matem\' atico\\ Fa\-cul\-tad de Ciencias\\ Universidad de Granada\\
18071 Granada, Spain} 
\email{alaminos@ugr.es, jlizana@ugr.es, mgodoy@ugr.es, avillena@ugr.es}

\begin{abstract}
Let $\mathcal{M}$ be a von Neumann algebra, and let $0<p,q\le\infty$.
Then the space $\Hom_\mathcal{M}(L^p(\mathcal{M}),L^q(\mathcal{M}))$ of all
right $\mathcal{M}$-module homomorphisms from $L^p(\mathcal{M})$
to $L^q(\mathcal{M})$ is a reflexive subspace of 
the space of all continuous linear maps from $L^p(\mathcal{M})$ to $L^q(\mathcal{M})$.
Further, the space $\Hom_\mathcal{M}(L^p(\mathcal{M}),L^q(\mathcal{M}))$ is hyperreflexive in each of
the following cases:
(i) $1\le q<p\le\infty$;
(ii) $1\le p,q\le\infty$ and $\mathcal{M}$ is injective, in which case the hyperreflexivity constant is at most $8$.
\end{abstract}


\thanks{
The authors were supported by project PGC2018-093794-B-I00 (MCIU/AEI/FEDER, UE),   
Junta de Andaluc\'{\i}a grant FQM-185, and
Proyectos I+D+i del programa operativo FEDER-Andaluc\'{\i}a  A-FQM-48-UGR18.
The third named author was supported by Contrato Predoctoral FPU, 
Plan propio de Investigaci\'on y Transferencia 2018, University of Granada.
}

\maketitle

\section*{Introduction}

Let $\mathcal{A}$ be a closed subalgebra of the algebra
$B(\mathcal{H})$ of all continuous linear operators on the Hilbert space $\mathcal{H}$.
Then $\mathcal{A}$ is called reflexive if
\begin{equation*}
\mathcal{A}=
\bigl\{
T\in B(\mathcal{H}) : e^\perp Te=0 \ (e\in\text{lat}\mathcal{A})
\bigr\},
\end{equation*}
where 
$\text{lat}\mathcal{A}=\{e\in B(\mathcal{H})\text{ projection} : e^\perp T e=0 \ (T\in\mathcal{A})\}$
is the set of all projections onto the $\mathcal{A}$-invariant subspaces of $\mathcal{H}$.
The double commutant theorem shows that each von Neumann algebra on $\mathcal{H}$ is certainly reflexive. 
The algebra $\mathcal{A}$ is called hyperreflexive if the above condition on $\mathcal{A}$ is strengthened by
requiring that there is a distance estimate
\begin{equation*}
\dist(T,\mathcal{A})\le 
C\sup\bigl\{\Vert e^\perp Te\Vert : e\in\text{lat}\mathcal{A}\bigr\}
\quad
(T\in B(\mathcal{H}))
\end{equation*}
for some constant $C$. 
The inequality 
\[
\sup\bigl\{\Vert e^\perp Te\Vert : e\in\text{lat}\mathcal{A}\bigr\}
\le
\dist(T,\mathcal{A})
\quad
(T\in B(\mathcal{H}))
\]
is always true and elementary.
This quantitative version of reflexivity was introduced by Arveson \cite{Ar} and has proven to be a powerful
tool when it is available.
Christensen \cite{Chr1, Chr2, Chr3} showed that many von Neumann algebras are hyperreflexive by relating the 
hyperreflexivity to the vanishing of certain cohomology group.
Notably each injective von Neumann algebra $\mathcal{M}$ on the Hilbert space $\mathcal{H}$ is hyperreflexive and
\begin{equation*}
\dist(T,\mathcal{M})\le 
4\sup\bigl\{\Vert e^\perp Te\Vert : e\in\mathcal{M}\text{ projection}\bigr\}
\quad
(T\in B(\mathcal{H}))
\end{equation*}
(see \cite[Theorem~2.3]{Chr1} and \cite[p. 340]{DS}).

Both notions, reflexivity and hyperreflexivity, were extended to subspaces of $B(\mathcal{X},\mathcal{Y})$,
the Banach space of all continuous linear maps from the Banach space $\mathcal{X}$ to the Banach space $\mathcal{Y}$.
Following Loginov and Shulman \cite{LS}, 
a closed linear subspace $\mathcal{S}$ of $B(\mathcal{X},\mathcal{Y})$ is called
reflexive if
\begin{equation*}
\mathcal{S}=
\bigl\{
T\in B(\mathcal{X},\mathcal{Y}) :
T(x)\in\overline{\{S(x) : S\in\mathcal{S}\}} \, (x\in\mathcal{X})
\bigr\}.
\end{equation*}
In accordance with Larson \cite{L3,L4},
$\mathcal{S}$ is called hyperreflexive if there exists a constant $C$ such that
\begin{equation*}
\dist(T,\mathcal{S})\le
C
\sup_{x\in\mathcal{X}, \, \Vert x\Vert\le 1}
\inf
\bigl\{\Vert T(x)-S(x)\Vert : S\in\mathcal{S}
\bigr\}
\quad
(T\in B(\mathcal{X},\mathcal{Y})),
\end{equation*}
and the optimal constant is called the hyperreflexivity constant of $\mathcal{S}$.
The inequality
\begin{equation*}
\sup_{x\in\mathcal{X}, \, \Vert x\Vert\le 1}
\inf
\bigl\{\Vert T(x)-S(x)\Vert : S\in\mathcal{S}
\bigr\}
\le
\dist(T,\mathcal{S})
\quad
(T\in B(\mathcal{X},\mathcal{Y})).
\end{equation*}
is always true.

The ultimate objective of this paper is to study the hyperreflexivity of the space 
$\Hom_\mathcal{M}(L^p(\mathcal{M}),L^q(\mathcal{M}))$ of all (automatically continuous) 
right  $\mathcal{M}$-module homomorphisms from $L^p(\mathcal{M})$ to $L^q(\mathcal{M})$ 
for a von Neumann algebra $\mathcal{M}$.
The non-commutative $L^p$-spaces that we consider throughout are those introduced by Haagerup 
(see \cite{H0,PX,T}).
For each $0<p\le\infty$, the space $L^p(\mathcal{M})$ is a contractive Banach $\mathcal{M}$-bimodule
or a contractive $p$-Banach $\mathcal{M}$-bimodule according to $1\le p$ or $p<1$, and
we will focus on the right $\mathcal{M}$-module structure of $L^p(\mathcal{M})$. 

Our method relies in the analysis  of a continuous bilinear map $\varphi\colon\mathcal{A}\times\mathcal{A}\to\mathcal{X}$,
for a $C^*$-algebra $\mathcal{A}$ and a normed space $\mathcal{X}$, through the knowledge of the
constant $\sup\{\Vert\varphi(a,b)\Vert : a,b\in\mathcal{A}_+\text{ contractions, } ab=0\}$,
alternatively, the constant $\sup\{\Vert\varphi(e,e^\perp)\Vert : e\in\mathcal{A}_+\text{ projection}\}$
in the case where $\mathcal{A}$ is unital and has real rank zero.
This is done in Section \ref{s1}. 

In Section \ref{s2} we prove that,
for each $0<p,q\le\infty$,
each right $\mathcal{M}$-module homomorphism from $L^p(\mathcal{M})$ to $L^q(\mathcal{M})$
is automatically continuous and that the space $\Hom_\mathcal{M}(L^p(\mathcal{M}),L^q(\mathcal{M}))$ of all 
right $\mathcal{M}$-module homomorphisms is a reflexive subspace of $B(L^p(\mathcal{M}),L^q(\mathcal{M}))$
(the notion of reflexivity makes perfect sense for subspaces of operators between quasi-Banach spaces).

Section \ref{s3} is devoted to study the hyperreflexivity of 
$\Hom_\mathcal{M}(L^p(\mathcal{M}),L^q(\mathcal{M}))$ for $1\le p,q\le\infty$.
The space $B(L^p(\mathcal{M}),L^q(\mathcal{M}))$ is a Banach
$\mathcal{M}$-bimodule for the operations specified by
\begin{equation*}
(aT)(x)=T(xa),
\quad
(Ta)(x)=T(x)a
\end{equation*}
for all $T\in B(L^p(\mathcal{M}),L^q(\mathcal{M}))$, $a\in\mathcal{M}$, and $x\in L^p(\mathcal{M})$
(note that the left $\mathcal{M}$-module structure of both $L^p(\mathcal{M})$ and $L^q(\mathcal{M})$ is disregarded),
and we will prove that there is a distance estimate
\[
\dist 
\bigl(T,\Hom_\mathcal{M}(L^p(\mathcal{M}),L^q(\mathcal{M}))\bigr)\le
C
\sup\bigl\{\Vert e^\perp Te\Vert : e\in\mathcal{M}\text{ projection}\bigr\}
\]
for each $T\in B(L^p(\mathcal{M}),L^q(\mathcal{M}))$
in each of the following cases:
\begin{enumerate}
\item[(i)]
$1\le q<p\le\infty$, in which case the constant $C$ can be chosen to depend on $p$ and $q$, and not on $\mathcal{M}$;
\item[(ii)]
$1\le p,q\le \infty$ and
$\mathcal{M}$ is injective, in which case the constant $C$ can be taken to be $8$.
\end{enumerate}
Further,
\begin{gather*}
\sup\bigl\{\Vert e^\perp Te\Vert : e\in\mathcal{M}\text{ projection}\bigr\} \\
\le
\sup_{x\in L^p(\mathcal{M}), \, \Vert x\Vert_p\le 1}
\inf
\bigl\{\Vert T(x)-\Phi(x)\Vert_q : \Phi\in\Hom_\mathcal{M}(L^p(\mathcal{M}),L^q(\mathcal{M}))
\bigr\},
\end{gather*}
and thus, in both cases, it turns out that the space $\Hom_\mathcal{M}(L^p(\mathcal{M}),L^q(\mathcal{M}))$ is hyperreflexive.

It is perhaps worth remarking that most of the discussion of reflexivity and hyperreflexivity is accomplished 
for continuous homomorphisms between modules over a $C^*$-algebra.

Throughout this paper 
we write $\mathcal{X}^*$ for the dual of  a Banach space $\mathcal{X}$ and 
$\langle\cdot,\cdot\rangle$ for the duality between $\mathcal{X}$ and $\mathcal{X}^*$.

\section{Analysing bilinear maps through orthogonality}\label{s1}

Goldstein proved in \cite{Go} (albeit with sesquilinear functionals) that, 
for each $C^*$-algebra $\mathcal{A}$, every continuous bilinear functional 
$\varphi\colon\mathcal{A}\times\mathcal{A}\to\mathbb{C}$ with the property that
$\varphi(a,b)=0$ whenever $a,b\in\mathcal{A}_{sa}$ satisfy $ab=0$ can be represented in the form
$\varphi(a,b)=\omega_1(ab)+\omega_2(ba)$ $(a,b\in\mathcal{A})$ for some $\omega_1,\omega_2\in\mathcal{A}^*$.
Independently, it was shown in \cite{Vil0} that if $\mathcal{A}$ is a 
$C^*$-algebra or the group algebra $L^1(G)$ of a locally compact group $G$, 
then every continuous bilinear functional
$\varphi\colon\mathcal{A}\times\mathcal{A}\to\mathbb{C}$ with the property that 
$\varphi(a,b)=0$ whenever $a,b\in\mathcal{A}$ are such that $ab=0$ necessarily satisfies
the condition $\varphi(ab,c)=\varphi(a,bc)$ $(a,b,c\in\mathcal{A})$, 
which in turn implies the existence of $\omega\in\mathcal{A}^*$ such that $\varphi(a,b)=\omega(ab)$ $(a,b\in\mathcal{A})$.
Actually, \cite{Vil} gives more, namely, 
the norms $\Vert\varphi(ab,c)-\varphi(a,bc)\Vert$ with $a,b,c\in\mathcal{A}$ 
can be estimated through the constant 
$\sup\{\Vert\varphi(a,b)\Vert : a,b\in\mathcal{A}, \, \Vert a\Vert=\Vert b\Vert=1, \, ab=0\}$. 
This property has proven to be useful to study the hyperreflexivity
of the spaces of derivations and continuous cocycles on $\mathcal{A}$ (see
\cite{Vil1, Vil2, Sam1, SS2, SS}). This section provides an improvement of the above mentioned property
in the case of $C^*$-algebras,
and this will be used later to study the hyperreflexivity of the space
$\Hom_\mathcal{A}(\mathcal{X},\mathcal{Y})$ of all continuous module homomorphisms between
the Banach right $\mathcal{A}$-modules $\mathcal{X}$ and $\mathcal{Y}$.

\begin{theorem}\label{t1946}
Let $\mathcal{A}$ be a $C^*$-algebra, 
let $\mathcal{Z}$ be a normed space, 
let $\varphi\colon\mathcal{A}\times\mathcal{A}\to\mathcal{Z}$ be a continuous bilinear map, and
let the constant $\varepsilon\ge 0$ be such that
\begin{equation*}
a,b\in\mathcal{A}_+, \ ab=0 \ \Longrightarrow \ \Vert\varphi(a,b)\Vert\le\varepsilon\Vert a\Vert\Vert b\Vert.
\end{equation*}
Suppose that $(e_j)_{j\in J}$ is a net in $\mathcal{A}$ such that $(e_j)_{j\in J}$ converges to 
$1_{\mathcal{A}^{**}}$ in $\mathcal{A}^{**}$ with respect to the weak* topology.
Then,
for each $a\in\mathcal{A}$,
the nets $(\varphi(a,e_j))_{j\in J}$ and $(\varphi(e_j,a))_{j\in J}$ converge in $\mathcal{Z}^{**}$ 
with respect to the weak* topology and
\begin{equation*}
\bigl\Vert\lim_{j\in J}\varphi(a,e_j)-\lim_{j\in J}\varphi(e_j,a)\bigr\Vert\le 8\varepsilon\Vert a\Vert.
\end{equation*}
In particular, if $\mathcal{A}$ is unital, then
\begin{equation*}
\Vert\varphi(a,1_\mathcal{A})-\varphi(1_\mathcal{A},a)\Vert\le 8\varepsilon\Vert a\Vert\quad (a\in\mathcal{A}).
\end{equation*}
\end{theorem}

\begin{proof}	
First, we regard $\varphi$ as a continuous bilinear map with values in $\mathcal{Z}^{**}$.
By applying \cite[Theorem~2.3]{JKR} to $\mathcal{A}$ acting on the Hilbert space of its universal representation, 
we obtain that  $\varphi$ extends uniquely, without change of norm, to a continuous bilinear map 
$\psi\colon\mathcal{A}^{**}\times\mathcal{A}^{**}\to\mathcal{Z}^{**}$ 
which is separately weak* continuous.

Now, since $(e_j)_{j\in J}\to 1_{\mathcal{A}^{**}}$ with respect to the weak* topology and $\psi$ 
is separately weak* continuous, we see that, for each $a\in\mathcal{A}$, the nets
$(\varphi(a,e_j))_{j\in J}$
and 
$(\varphi(e_j,a))_{j\in J}$
converge to $\psi(a,1_{\mathcal{A}^{**}})$ and
$\psi(1_{\mathcal{A}^{**}},a)$, respectively, with
respect to the weak* topology of $\mathcal{Z}^{**}$.
Consequently, the proof of the theorem is completed by showing that
\begin{equation}\label{e1058}
\bigl\Vert\psi(a,1_{\mathcal{A}^{**}})-\psi(1_{\mathcal{A}^{**}},a)\bigr\Vert
\le
8\varepsilon\Vert a\Vert
\quad
(a\in\mathcal{A}).
\end{equation}

Our next objective is to prove \eqref{e1058}.

We begin with the case $a\in\mathcal{A}_+$.
For this purpose, 
we fix $a\in\mathcal{A}_+$ with $\Vert a\Vert\le 1$ and,
for each $0<\alpha<1$, we claim that
\begin{equation}\label{e1530}
\bigl\Vert\psi\bigl(\chi_{[0,\alpha]}(a),\chi_{]\alpha,1]}(a)\bigr)\bigr\Vert
\le \varepsilon
\end{equation}
and
\begin{equation}\label{e1531}
\bigl\Vert\psi\bigl(\chi_{]\alpha,1]}(a),\chi_{[0,\alpha]}(a)\bigr)\bigr\Vert
\le \varepsilon.
\end{equation}
We use the notation $\chi_\Delta$ for the characteristic function of a subset $\Delta$ of $[0,1]$.
We choose decreasing sequences of real numbers $(\alpha_n)$ and $(\beta_n)$ with $\alpha<\alpha_n<\beta_n<1$
$(n\in\mathbb{N})$ and $\lim\alpha_n=\lim\beta_n=\alpha$. For each $n\in\mathbb{N}$, we define continuous functions
$f_n,g_n\colon[0,1]\to\mathbb{R}$ by
\[
f_n(t)= 
    \begin{cases}
1 & \text{if $0\le t\le\alpha$,}\\
\frac{t-\alpha_n}{\alpha-\alpha_n} & \text{if $\alpha\le t\le\alpha_n$,}\\
0 & \text{if $\alpha_n\le t\le 1$,}
\end{cases}
\quad
g_n(t)=
\begin{cases}
0 & \text{if $0\le t\le\alpha_n$,}\\
\frac{t-\alpha_n}{\beta_n-\alpha_n} & \text{if $\alpha_n\le t\le\beta_n$,}\\
1 & \text{if $\beta_n\le t\le 1$}.
    \end{cases}
\]
Then the sequences $(f_n)$ and $(g_n)$ are uniformly bounded and they
converge pointwise in $[0,1]$ to $\chi_{[0,\alpha]}$ and $\chi_{]\alpha,1]}$, respectively.
From a basic property of the Borel functional calculus on the von Neumann algebra $\mathcal{A}^{**}$,
it follows that the sequences $\bigl(f_n(a)\bigr)$ and $\bigl(g_n(a)\bigr)$ converge with 
respect to the weak operator topology to $\chi_{[0,\alpha]}(a)$ and $\chi_{]\alpha,1]}(a)$, respectively. 
Since the sequences $\bigl(f_n(a)\bigr)$ and $\bigl(g_n(a)\bigr)$ are bounded and the weak operator topology 
of $\mathcal{A}^{**}$ coincides with the weak* topology on any bounded 
subset of $\mathcal{A}^{**}$, we conclude that $\bigl(f_n(a)\bigr)$ and $\bigl(g_n(a)\bigr)$ converge with 
respect to the weak* topology to $\chi_{[0,\alpha]}(a)$ and $\chi_{]\alpha,1]}(a)$, respectively. Hence
\begin{equation}\label{e1532}
\psi\bigl(\chi_{[0,\alpha]}(a),\chi_{]\alpha,1]}(a)\bigr)=
\lim_{n\to\infty}\lim_{m\to\infty}\varphi\bigl(f_m(a),g_n(a)\bigr)
\end{equation}
and
\begin{equation}\label{e1533}
\psi\bigl(\chi_{]\alpha,1]}(a),\chi_{[0,\alpha]}(a)\bigr)=
\lim_{n\to\infty}\lim_{m\to\infty}\varphi\bigl(g_n(a),f_m(a)\bigr)
\end{equation}
On the other hand, if $m,n\in\mathbb{N}$ with $m\ge n$, then we have $f_mg_n=0$, so that 
\[
f_m(a)g_n(a)=g_n(a)f_m(a)=0.
\]
Further, $f_m(a),g_n(a)\in\mathcal{A}_+$ and $\Vert f_m(a)\Vert,\Vert g_n(a)\Vert\le 1$.
Therefore, by hypothesis, we have
\begin{equation}\label{e1534}
\bigl\Vert\varphi\bigl(f_m(a),g_n(a)\bigr)\bigr\Vert\le\varepsilon
\end{equation}
and
\begin{equation}\label{e1535}
\bigl\Vert\varphi\bigl(g_n(a),f_m(a)\bigr)\bigr\Vert\le\varepsilon
\end{equation}
for all $m,n\in\mathbb{N}$ with $m\ge n$.
From \eqref{e1532} and \eqref{e1534} we deduce \eqref{e1530}, 
while \eqref{e1533} and \eqref{e1535} give \eqref{e1531}.
For each $n\in\mathbb{N}$, let $h_n\colon[0,1]\to\mathbb{R}$ be the bounded Borel function defined by
\[
h_n=\frac{1}{n+1}\sum_{k=1}^n\chi_{]k/(n+1),1]}.
\]
We also consider the sequence in $\mathcal{A}^{**}$ given by $\bigl(h_n(a)\bigr)$.
Since the sequence $(h_n)$ converges uniformly on $[0,1]$ to the identity map on $[0,1]$,
it follows that $\bigl(h_n(a)\bigr)\to a$ with respect to the norm topology.
Thus
\begin{equation}\label{e1536}
\begin{split}
&\psi(a,1_{\mathcal{A}^{**}})-\psi(1_{\mathcal{A}^{**}},a)=\lim_{n\to\infty}\Bigl(\psi\bigl(h_n(a),1_{\mathcal{A}^{**}}\bigr)-
\psi\bigl(1_{\mathcal{A}^{**}},h_n(a)\bigr)\Bigr)\\
&=
\lim_{n\to\infty}\frac{1}{n+1}
\sum_{k=1}^n\Bigl(\psi\bigl(\chi_{]\frac{k}{n+1},1]}(a),1_{\mathcal{A}^{**}}\bigr)-
\psi\bigl(1_{\mathcal{A}^{**}},\chi_{]\frac{k}{n+1},1]}(a)\bigr)\Bigr).
\end{split}
\end{equation}
Further, for each $k\in\{1,\ldots,n\}$, we have
\[
1_{\mathcal{A}^{**}}=\chi_{[0,k/(n+1)]}(a)+\chi_{]k/(n+1),1]}(a),
\]
so that
\begin{multline*}
\psi\bigl(\chi_{]k/(n+1),1]}(a),1_{\mathcal{A}^{**}}\bigr)-
\psi\bigl(1_{\mathcal{A}^{**}},\chi_{]k/(n+1),1]}(a)\bigr) \\
=
\psi\bigl(\chi_{]k/(n+1),1]}(a),\chi_{[0,k/(n+1)]}(a)\bigr)-
\psi\bigl(\chi_{[0,k/(n+1)]}(a),\chi_{]k/(n+1),1]}(a)\bigr),
\end{multline*}
and the inequalities \eqref{e1530} and \eqref{e1531} then give
\begin{equation}\label{e1537}
\bigl\Vert\psi\bigl(\chi_{]k/(n+1),1]}(a),1_{\mathcal{A}^{**}}\bigr)-
\psi\bigl(1_{\mathcal{A}^{**}},\chi_{]k/(n+1),1]}(a)\bigr)\bigr\Vert\le 2\varepsilon.
\end{equation}
From \eqref{e1536} and \eqref{e1537} it may be concluded that
\begin{equation}\label{e1011}
\bigl\Vert\psi(a,1_{\mathcal{A}^{**}})-\psi(1_{\mathcal{A}^{**}},a)\bigr\Vert
\le
\lim_{n\to\infty}\frac{n2\varepsilon}{n+1}=2\varepsilon.
\end{equation}

Now suppose that $a\in\mathcal{A}_{\textup{sa}}$.
Then we can write $a=a_+-a_-$, where $a_+,a_-\in\mathcal{A}_+$ are mutually orthogonal, and \eqref{e1011} gives
\begin{align*}
\bigl\Vert\psi(a,1_{\mathcal{A}^{**}})-\psi(1_{\mathcal{A}^{**}},a)\bigr\Vert
& \le
\bigl\Vert\psi(a_+,1_{\mathcal{A}^{**}})-
\psi(1_{\mathcal{A}^{**}},a_+)\bigr\Vert \\
& \quad {}+
\bigl\Vert\psi(a_-,1_{\mathcal{A}^{**}})-
\psi(1_{\mathcal{A}^{**}},a_-)\bigr\Vert
\\ 
& \le
2\varepsilon\Vert a_+\Vert+2\varepsilon\Vert a_-\Vert
\\ 
& \le
4\varepsilon\max\{\Vert a_+\Vert,\Vert a_-\Vert\}
=
4\varepsilon\Vert a\Vert.
\end{align*}

Finally take $a\in\mathcal{A}$, and write 
\begin{equation}\label{z1} 
a=\Re a+i\Im a,
\end{equation} 
where
\[
\Re a=\frac{1}{2}(a^*+a), \, 
\Im a=\frac{i}{2}(a^*-a)\in\mathcal{A}_{\textup{sa}},
\]
and, further,
$\Vert\Re a\Vert,\Vert\Im a\Vert\le\Vert a\Vert$.
From what has previously been proved, it follows that
\begin{align*}
\bigl\Vert\psi(a,1_{\mathcal{A}^{**}})-\psi(1_{\mathcal{A}^{**}},a)\bigr\Vert
& \le
\bigl\Vert\psi(\Re a,1_{\mathcal{A}^{**}})-\psi(1_{\mathcal{A}^{**}},\Re a)\bigr\Vert \\
& \quad {}+
\bigl\Vert\psi(\Im a,1_{\mathcal{A}^{**}})-\psi(1_{\mathcal{A}^{**}},\Im a) \bigr\Vert
\\
& \le
4\varepsilon\Vert\Re a\Vert+4\varepsilon\Vert\Im a\Vert
\\ 
& \le
8\varepsilon\Vert a\Vert.
\end{align*}
This gives \eqref{e1058} and completes the proof of the theorem.
\end{proof}

\begin{theorem}\label{t1945}
Let $\mathcal{A}$ be a unital $C^*$-algebra of real rank zero,
let $\mathcal{Z}$ be a topological vector space, and
let $\varphi\colon\mathcal{A}\times\mathcal{A}\to\mathcal{Z}$ be a continuous bilinear map.
\begin{enumerate}
\item[(i)]
Suppose that 
\begin{equation*}
e\in\mathcal{A}\text{ projection} \ \Longrightarrow \ \varphi(e,e^\perp)=0.
\end{equation*}
Then
\begin{equation*}
\varphi(a,1_\mathcal{A})=\varphi(1_\mathcal{A},a)
\quad (a\in \mathcal{A}).
\end{equation*}
\item[(ii)]
Suppose that $\mathcal{Z}$ is a normed space and let the constant $\varepsilon\ge 0$ be such that
\begin{equation*}
e\in\mathcal{A}\text{ projection} \ \Longrightarrow \ \Vert\varphi(e,e^\perp)\Vert\le\varepsilon.
\end{equation*}
Then
\begin{equation*}
\left\Vert\varphi(a,1_\mathcal{A})-\varphi(1_\mathcal{A},a)\right\Vert\le 8\varepsilon\Vert a\Vert
\quad (a\in \mathcal{A}).
\end{equation*}
\end{enumerate}	
\end{theorem}

\begin{proof}
Let $a\in\mathcal{A}_+$, and 
assume that $a$ has finite spectrum, say $\{\alpha_1,\ldots,\alpha_n\}$. 
Of course, we can suppose that $0\le\alpha_1<\cdots<\alpha_n=\Vert a\Vert$.
This implies that $a$ can be written in the form
\[
a=\sum_{k=1}^n\alpha_k p_k,
\]
where $p_1,\ldots,p_n\in\mathcal{A}$ are mutually orthogonal projections 
(specifically, the projection $p_k$ is defined by using the continuous functional calculus for $a$ by
$p_k=\chi_{\{\alpha_k\}}(a)$ for each $k\in\{1,\ldots,n\}$ because $\chi_{\{\alpha_k\}}$ is a continuous function on
the spectrum of $a$, being this set finite).
In case (i), we have
\begin{equation}\label{e89}
\begin{split}
\varphi(a,1_\mathcal{A})-\varphi(1_\mathcal{A},a)
&=
\sum_{k=1}^n\alpha_k
\bigl(\varphi(p_k,1_\mathcal{A})-\varphi(1_\mathcal{A},p_k)\bigr)\\
&=
\sum_{k=1}^n\alpha_k
\bigl(\varphi(p_k,p_{k}^\perp+p_k)-\varphi(p_{k}^\perp+p_k,p_k)\bigr)\\
&=
\sum_{k=1}^n\alpha_k
\bigl(\varphi(p_k,p_{k}^\perp)-\varphi(p_{k}^\perp,p_k)\bigr) =0.
\end{split}
\end{equation}
In case (ii),
we define real numbers $\lambda_1,\ldots,\lambda_n\in[0,\infty[$ and projections $e_1,\ldots,e_n\in\mathcal{A}$ by
\begin{align*}
\lambda_1 & =\alpha_1,\\
\lambda_k & =\alpha_k-\alpha_{k-1} \quad (1<k\le n),\\
\intertext{and}
e_k & =p_k+\cdots+p_n\quad (1\le k\le n).
\end{align*}
It is a simple matter to check that
\[
a=\sum_{k=1}^n\lambda_k e_k.
\]
For each $k\in\{1,\ldots,n\}$, we have
\begin{equation*}
\begin{split}
\left\Vert\varphi(e_k,1_\mathcal{A})-\varphi(1_\mathcal{A},e_k)\right\Vert
&=
\left\Vert\varphi(e_k, e_{k}^\perp+e_k)-\varphi({e_k}^\perp+e_k,e_k)\right\Vert\\
&=
\left\Vert\varphi(e_k,e_{k}^\perp)-\varphi(e_{k}^\perp,e_k)\right\Vert\\
&\le
\left\Vert\varphi(e_k,e_{k}^\perp)\right\Vert+
\left\Vert\varphi(e_{k}^\perp,e_k) \right\Vert 
\le 2\varepsilon.
\end{split}
\end{equation*}
We thus get
\begin{equation}\label{e123}
\begin{split}
\left\Vert\varphi(a,1_\mathcal{A})-\varphi(1_\mathcal{A},a)\right\Vert
&=
\left\Vert\sum_{k=1}^n
\lambda_k\bigl(\varphi(e_k,1_\mathcal{A})-\varphi(1_\mathcal{A},e_k)\bigr)\right\Vert\\
&\le
\sum_{k=1}^n\lambda_k
\bigl\Vert\varphi(e_k,1_\mathcal{A})-\varphi(1_\mathcal{A},e_k)\bigr\Vert\\
&\le
\sum_{k=1}^n\lambda_k 2\varepsilon
=
2 \varepsilon \alpha_n 
= 2\varepsilon \Vert a\Vert .
\end{split}
\end{equation}
	
Now suppose that $a\in\mathcal{A}_{\textup{sa}}$ and that $a$ has finite spectrum.
Then we can write $a=a_+-a_-$, where $a_+$, $a_-\in\mathcal{A}_+$ are mutually orthogonal. 
Since $a_+=f(a)$ and $a_-=g(a)$, where $f$, $g\colon\mathbb{R}\to\mathbb{R}$ are the continuous functions defined by
	\[
	f(t)=\max\{t,0\}, \quad g(t)=\max\{-t,0\}, \quad (t\in\mathbb{R}),
	\]
it follows that both $a_+$ and $a_-$ have finite spectra.
In case (i), \eqref{e89} gives
\begin{equation}\label{e90}
\varphi(a,1_\mathcal{A})-\varphi(1_\mathcal{A},a)=
\bigl(\varphi(a_+,1_\mathcal{A})-\varphi(1_\mathcal{A},a_+)\bigr)-
\bigl(\varphi(a_-,1_\mathcal{A})-\varphi(1_\mathcal{A},a_-)\bigr)=0.
\end{equation}
In case (ii),
on account of \eqref{e123}, we have
\begin{equation}\label{e91}
\begin{split}
\left\Vert\varphi(a,1_\mathcal{A})-\varphi(1_\mathcal{A},a)\right\Vert
& \le
\left\Vert\varphi(a_+,1_\mathcal{A})-\varphi(1_\mathcal{A},a_+)\right\Vert \\
&\quad {}+ \left\Vert\varphi(a_-,1_\mathcal{A})-\varphi(1_\mathcal{A},a_-)\right\Vert \\
& \le 2\varepsilon\Vert a_+\Vert+2\varepsilon\Vert a_-\Vert
\\ 
& \le 
4\varepsilon\max\{\Vert a_+\Vert,\Vert a_-\Vert\}=
4\varepsilon\Vert a\Vert.
\end{split}
\end{equation}

Let $a$ be an arbitrary element of $\mathcal{A}_{\textup{sa}}$.
Since $\mathcal{A}$ has real rank zero, it follows that there exists a sequence $(a_n)$ in 
$\mathcal{A}_{\textup{sa}}$ such that each $a_n$ has finite spectrum and $(a_n)\to a$ with respect
to the norm topology.
In case (i), \eqref{e90} gives
\begin{equation*}
\varphi(a,1_\mathcal{A})-\varphi(1_\mathcal{A},a)=
\lim_{n\to\infty}
\bigl(
\varphi(a_n,1_\mathcal{A})-\varphi(1_\mathcal{A},a_n)
\bigr)=0.
\end{equation*}
In case (ii), 
from \eqref{e91} it follows that 
\[
\Vert\varphi(a_n,1_\mathcal{A})-\varphi(1_\mathcal{A},a_n)\Vert\le 4\varepsilon\Vert a_n\Vert\quad (n\in\mathbb{N}),
\]
and the continuity of $\varphi$ now yields
\[
\begin{split}
\left\Vert\varphi(a,1_\mathcal{A})-\varphi(1_\mathcal{A},a) \right\Vert & =
\lim_{n \to \infty} \left\Vert\varphi(a_n,1_\mathcal{A})-\varphi(1_\mathcal{A},a_n)\right\Vert \\
& \le 4\varepsilon\lim_{n \to \infty} \Vert a_n\Vert=
4\varepsilon\Vert a\Vert.
\end{split}
\]

Finally set  $a\in\mathcal{A}$, and write $a=\Re a+i\Im a$ as in \eqref{z1}.
From the previous step we see that,
in case (i),
\begin{equation*}
\varphi(a,1_\mathcal{A})-\varphi(1_\mathcal{A},a)=
\varphi(\Re a,1_\mathcal{A})-\varphi(1_\mathcal{A},\Re a) 
+
\varphi(\Im a,1_\mathcal{A})-\varphi(1_\mathcal{A},\Im a)=0,
\end{equation*}
and, in case (ii),
\begin{equation*}
\begin{split}
\Vert\varphi(a,1_\mathcal{A})-\varphi(1_\mathcal{A},a)\Vert
& \le \left\Vert\varphi(\Re a,1_\mathcal{A})-\varphi(1_\mathcal{A},\Re a)\right\Vert \\
& \quad {}+
\left\Vert\varphi(\Im a,1_\mathcal{A})-\varphi(1_\mathcal{A},\Im a) \right\Vert
\\ 
& \le 4\varepsilon\Vert\Re a\Vert+4\varepsilon\Vert\Im a\Vert
\le
8\varepsilon\Vert a\Vert,
\end{split}
\end{equation*}
giving the result.
\end{proof}

\begin{corollary}\label{t1947}
Let $\mathcal{A}$ be a $C^*$-algebra,
let $\mathcal{Z}$ be a normed space, 
let $\varphi\colon\mathcal{A}\times\mathcal{A}\to\mathcal{Z}$ be a continuous bilinear map, and
let the constant $\varepsilon\ge 0$ be such that
\begin{equation*}
a,b\in\mathcal{A}, \ ab=0 \ \Longrightarrow \ \Vert\varphi(a,b)\Vert\le\varepsilon\Vert a\Vert\Vert b\Vert.
\end{equation*}
Then
\begin{equation*}
\Vert\varphi(ab,c)-\varphi(a,bc)\Vert\le 8\varepsilon\Vert a\Vert\Vert b\Vert\Vert c\Vert
\quad (a,b,c\in\mathcal{A}).
\end{equation*}
Further, if either $\mathcal{A}$ is unital or $\mathcal{Z}$ is a dual Banach space, 
then there exists a continuous linear map $\Phi\colon\mathcal{A}\to\mathcal{Z}$ such that
\begin{equation*}
\Vert\varphi(a,b)-\Phi(ab)\Vert\le 8\varepsilon\Vert a\Vert\Vert b\Vert
\quad (a,b\in\mathcal{A})
\end{equation*}
and $\Vert\Phi\Vert\le\Vert\varphi\Vert$.
\end{corollary}

\begin{proof}
Set $a,c\in\mathcal{A}$, and consider the continuous bilinear map
$\psi\colon\mathcal{A}\times\mathcal{A}\to\mathcal{Z}$ 
defined by
\[
\psi(u,v)=\varphi(au,vc) 
\quad 
(u,v\in\mathcal{A}).
\]
If $u,v\in A$ are such that $uv=0$, then
$au,vc\in\mathcal{A}$ and $(au)(vc)=0$, which gives 
\[
\Vert\psi(u,v)\Vert=
\Vert\varphi(au,vc)\Vert\le
\varepsilon\Vert au\Vert\Vert vc\Vert\le
\varepsilon\Vert a\Vert\Vert c\Vert  \Vert u\Vert\Vert v\Vert
\]
by hypothesis.

Let $(e_i)_{i\in I}$ be an approximate identity for $\mathcal{A}$ of bound one.
Since the net $(e_i)_{i\in I}$ is bounded, it has a subnet
$(e_j)_{j\in J}$ which converges to an element $E\in\mathcal{A}^{**}$ with
respect to the weak* topology. We claim that $E=1_{\mathcal{A}^{**}}$.
Indeed, let $a\in\mathcal{A}$. Then $(ae_j)_{j\in J}\to a$ with respect to the 
norm topology and, further, $(ae_j)_{j\in J}\to aE$ with respect to the weak*
topology. Thus $aE=a$. From the weak* density of $\mathcal{A}$ in $\mathcal{A}^{**}$ 
and the separate weak* continuity of the product of $\mathcal{A}^{**}$, we conclude 
that $AE=A$ for each $A\in\mathcal{A}^{**}$, hence that $E=1_{\mathcal{A}^{**}}$, as claimed.

From Theorem~\ref{t1946} we deduce that,
for each $b\in\mathcal{A}$,
the nets $(\psi(b,e_j))_{j\in J}$ and $(\psi(e_j,b))_{j\in J}$ converge in $\mathcal{Z}^{**}$
 with respect to the weak* topology and
\begin{equation}\label{e1829}
\bigl\Vert\lim_{j\in J}\psi(b,e_j)-\lim_{j\in J}\psi(e_j,b)\bigr\Vert\le 
8(\varepsilon\Vert a\Vert\Vert c\Vert)\Vert b\Vert.
\end{equation}
Since $(ae_j)_{j\in J}$ converges to $a$ and $(e_jc)_{j\in J}$ converges to $c$ in norm and $\varphi$ is continuous, it follows that
\[
\lim_{j\in J} \psi(b,e_j)= \lim_{j\in J}\varphi(ab,e_jc) = \varphi(ab,c)
\]
and
\[
\lim_{j\in J} \psi(e_j,b)= \lim_{j\in J} \varphi(ae_j,bc) = \varphi(a,bc)
\]
in norm for each $b\in\mathcal{A}$, which establishes the required inequality when combined with~\eqref{e1829}.
 
Now suppose that $\mathcal{A}$ is unital, and
define $\Phi\colon\mathcal{A}\to\mathcal{Z}$ by
\[
\Phi(a)=\varphi(a,1_\mathcal{A})
\quad
(a\in\mathcal{A}).
\]
Then $\Phi$ is a continuous linear map, and clearly $\Vert\Phi\Vert\le\Vert\varphi\Vert$.
For each $a$, $b\in\mathcal{A}$, we have
\[
\Vert\varphi(a,b)-\Phi(ab)\Vert=
\Vert\varphi(a,b1_\mathcal{A})-\varphi(ab,1_\mathcal{A})\Vert\le
8\varepsilon\Vert a\Vert\Vert b\Vert,
\]
as claimed.

Finally suppose that $\mathcal{Z}$ is the dual of a Banach space $\mathcal{Z}_*$.
Let $\mathcal{U}$ be an ultrafilter on $J$ containing the order filter on $J$.
It follows from the Banach-Alaoglu theorem that each bounded subset of $\mathcal{Z}$ is relatively
compact with respect to the weak* topology on $\mathcal{Z}$. Consequently, each bounded net
$(z_j)_{j\in J}$ in $\mathcal{Z}$ has a unique limit with respect to the weak* topology along the ultrafilter
$\mathcal{U}$, and we write $\lim_\mathcal{U} z_j$ for this limit.
Further, it is worth noting that
\begin{equation}\label{e1608}
\Bigl\Vert\lim_\mathcal{U}z_j\Bigr\Vert\le\lim_\mathcal{U}\Vert z_j\Vert.
\end{equation}
Indeed, for each $\zeta\in\mathcal{Z}_*$ such that $\Vert\zeta\Vert=1$, we have
$\vert\langle z_j,\zeta\rangle\vert\le\Vert z_j\Vert$ $(j\in J)$, and hence
\[
\Bigl\vert\Bigl\langle\lim_\mathcal{U}z_j,\zeta\Bigr\rangle\Bigr\vert=
\lim_\mathcal{U}\vert\langle z_j,\zeta\rangle\vert\le\lim_\mathcal{U}\Vert z_j\Vert,
\]
which establishes \eqref{e1608}.

For each $a\in\mathcal{Z}$, we have
\begin{equation}\label{e163}
\Vert\varphi(a,e_j)\Vert\le\Vert\varphi\Vert\Vert a\Vert
\quad
(j\in J),
\end{equation}
and hence the net $\left(\varphi(a,e_j)\right)_{j\in J}$ is bounded. Consequently, we can define
the map $\Phi\colon\mathcal{A}\to\mathcal{Z}$ by
\[
\Phi(a)=\lim_\mathcal{U}\varphi(a,e_j)
\quad
(a\in\mathcal{A}).
\]
The linearity of the limit along an ultrafilter on a topological
linear space gives the linearity of $\Phi$. 
Take $a\in\mathcal{A}$. 
From \eqref{e1608} and \eqref{e163} we deduce that
$\Vert\Phi(a)\Vert\le\Vert\varphi\Vert\Vert a\Vert$,
which gives the continuity of $\Phi$ and $\Vert\Phi\Vert\le\Vert\varphi\Vert$.
Now take  $a,b\in\mathcal{A}$. 
We have
\begin{equation}\label{z2}
\Vert\varphi(ab,e_j)-\varphi(a,be_j)\Vert\le 8\varepsilon\Vert a\Vert\Vert b\Vert
\quad
(j\in J).
\end{equation}
Since $(be_j)_{j\in J}\to b$ in norm, the continuity of $\varphi$ gives $(\varphi(a,be_j))_{j\in J}\to\varphi(a,b)$ in norm.
Since $\mathcal{U}$ refines the order filter on $J$ we see that
$\lim_\mathcal{U}\varphi(a,be_j)=\varphi(a,b)$. 
Taking the limit in \eqref{z2} along the ultrafilter $\mathcal{U}$, and using~\eqref{e1608}, we obtain
$\Vert\Phi(ab)-\varphi(a,b)\Vert
\le 8\varepsilon\Vert a\Vert\Vert b\Vert$,
as required.
\end{proof}

It is worth noting that Corollary~\ref{t1947} 
gives a sharper estimate for the constant of the strong property $\mathbb{B}$
of a $C^*$-algebra to the one given in \cite[Theorem~3.4]{SS}, where 
our constant $8$ is replaced by $384\pi^2(1+\sqrt{2})$. The hyperreflexivity constant given in \cite[Theorem~4.4]{SS} for $C^*$-algebras can be sharpened as well accordingly.

\section{Primary estimates and reflexivity}\label{s2}

\subsection{Homomorphisms between modules over a $C^*$-algebra}

Let $\mathcal{A}$ be a $C^*$-algebra, and
let $\mathcal{X}$ and $\mathcal{Y}$ be quasi-Banach right $\mathcal{A}$-modules.
For a linear map $T\colon\mathcal{X}\to\mathcal{Y}$ and $a\in\mathcal{A}$, 
define linear maps $aT$, $Ta\colon\mathcal{X}\to\mathcal{Y}$ by
\begin{equation}\label{eqm}
(aT)(x)=T(xa), 
\quad 
(Ta)(x)=T(x)a
\quad
(x\in\mathcal{X}).
\end{equation}
Then the space $B(\mathcal{X},\mathcal{Y})$ of all continuous linear maps from $\mathcal{X}$ to $\mathcal{Y}$ 
is a quasi-Banach $\mathcal{A}$-bimodule for the operations specified by \eqref{eqm}.
For $T\in B(\mathcal{X},\mathcal{Y})$,
let $\ad(T)\colon\mathcal{A}\to B(\mathcal{X},\mathcal{Y})$ denote the inner derivation implemented by $T$,
so that
\[
\ad(T)(a)=aT-Ta\quad(a\in\mathcal{A}).
\] 
It is clear that $T$ is a right $\mathcal{A}$-module homomorphism if and only if $\ad(T)=0$, 
and, in the case where $\mathcal{X}$ and $\mathcal{Y}$ are Banach $\mathcal{A}$-modules,
the constant
\begin{equation}\label{def}
\Vert\ad(T)\Vert
\end{equation} 
is intended to estimate the distance from $T$ to the space
$\Hom_\mathcal{A}(\mathcal{X},\mathcal{Y})$
of all continuous module homomorphisms from $\mathcal{X}$ to $\mathcal{Y}$.
This is actually very much in the spirit of \cite{Chr1, Chr2, Chr3}.
We will use several additional alternative ways to estimate the distance 
$\dist(T,\Hom_\mathcal{A}(\mathcal{X},\mathcal{Y}))$ that equally make sense, namely
\begin{equation}\label{alpha}
\alpha(T)=
\sup\bigl\{
\Vert e^\perp Te\Vert : e\in\mathcal{A} \text{ projection}
\bigr\},
\end{equation}
\begin{equation}\label{beta}
\beta(T)=
\sup\bigl\{
\Vert eTf\Vert : e,f\in\mathcal{A} \text{ projections, } ef=0\bigr\},
\end{equation}
\begin{equation}\label{gamma}
\gamma(T)=
\sup\bigl\{
\Vert aTb\Vert : a,b\in\mathcal{A}_+\text{ contractions, } \ ab=0\bigr\},
\end{equation}
\begin{equation}\label{delta}
\delta(T)=
\sup_{x\in\mathcal{X}, \, \Vert x\Vert\le 1}
\inf
\bigl\{\Vert T(x)-\Phi(x)\Vert : \Phi\in\text{Hom}_\mathcal{A}(\mathcal{X},\mathcal{Y})
\bigr\}.
\end{equation}
For \eqref{alpha}, the algebra is supposed to be unital, and this constant is actually very much in the spirit 
of the celebrated Arveson's distance formula~\cite{Arv}.
The significance of the constants~\eqref{alpha} and~\eqref{beta} for our purposes requires
that the $C^*$-algebra $\mathcal{A}$ to be sufficiently rich in projections (as does a $C^*$-algebra
of real rank zero) and they are in the spirit of \cite{L1,L2}. 
We take the constant \eqref{delta} in accordance with \cite{L3,L4}.

Throughout we suppose that the Banach $\mathcal{A}$-modules are contractive.
The statements of our results can be easily adapted to non-contractive Banach modules.

\begin{proposition}\label{p10}
Let $\mathcal{A}$ be a $C^*$-algebra,
let $\mathcal{X}$ and $\mathcal{Y}$ be Banach right $\mathcal{A}$-modules, and
let $T\colon \mathcal{X}\to\mathcal{Y}$ be a continuous linear map. 
\begin{enumerate}
\item[(i)]
$\beta(T)\le\gamma(T)\le\delta(T)\le\dist(T,\Hom_\mathcal{A}(\mathcal{X},\mathcal{Y}))$.
\item[(ii)]
$\gamma(T)\le\Vert\ad(T)\Vert\le 2\dist(T,\Hom_\mathcal{A}(\mathcal{X},\mathcal{Y}))$.
\item[(iii)]
If $\mathcal{A}$ is unital, then 
$\alpha(T)=\beta(T)$.
\item[(iv)]
If $\mathcal{A}$ is unital and has real rank zero, then
$\beta(T)=\gamma(T)$.
\end{enumerate}
\end{proposition}

\begin{proof}

(i)
This is immediate.

(ii)
Let $a,b\in\mathcal{A}_+$ contractions with $ab=0$. 
Then $aTb=(aT-Ta)b$ and therefore
$\Vert aTb\Vert\le\Vert aT-Ta\Vert\le\Vert\ad(T)\Vert$.
We thus get $\gamma(T)\le\Vert\ad(T)\Vert$.

Now take a sequence $(\Phi_n)$ in $\Hom_\mathcal{A}(\mathcal{X},\mathcal{Y})$ such that
\[
\dist 
\bigl(T,\Hom_\mathcal{A}(\mathcal{X},\mathcal{Y})\bigr)=
\lim_{n\to\infty}\Vert T-\Phi_n\Vert.
\]
Let $x\in\mathcal{X}$ and $a\in\mathcal{A}$ with $\Vert x\Vert=\Vert a\Vert=1$. Then
\begin{equation*}
\begin{split}
\left\Vert T(xa)-T(x)a \right\Vert
&= 
\left\Vert T(xa)-T(x)a-(\Phi_n(xa)-\Phi_n(x)a) \right\Vert \\
&\le
\left\Vert T(xa)-\Phi_n(xa)\Vert+\Vert(T(x)-\Phi_n(x))a \right\Vert \\
&\le
2\Vert T-\Phi_n\Vert , 
\end{split}
\end{equation*}
which gives 
$\Vert (aT-Ta)(x)\Vert\le 2\dist \bigl(T,\Hom_\mathcal{A}(\mathcal{X},\mathcal{Y})\bigr)$,
and the second inequality is proved.

(iii)
It suffices to prove that $\beta(T)\le\alpha(T)$.
Let $e,f\in\mathcal{A}$ mutually orthogonal projections. Then $e\le f^\perp$, so that
$eTf=e(f^\perp Tf)$ and therefore 
\[
\Vert eTf\Vert=\Vert e(f^\perp Tf)\Vert\le\Vert f^\perp T f\Vert\le\alpha(T),
\]
which implies that $\beta(T)\le\alpha(T)$, as required.

(iv)
It suffices to show that $\gamma(T)\le\beta(T)$.
Take $a,b\in\mathcal{A}_+$ mutually orthogonal contractions, and
the task is to show that $\Vert aTb\Vert\le\beta(T)$.

First, assume that both $a$ and $b$ have finite spectrum, say
$\{\alpha_1,\ldots,\alpha_m\}$ and $\{\beta_1,\ldots,\beta_n\}$ with
$0\le\alpha_1<\cdots<\alpha_m=\Vert a\Vert\le 1$ and
$0\le\beta_1<\cdots<\beta_n=\Vert b\Vert\le 1$, and write
\[
a=\sum_{j=1}^m\alpha_j p_j,
\quad
b=\sum_{k=1}^n\beta_k q_k,
\]
where $p_1,\ldots,p_m,q_1,\ldots,q_n\in\mathcal{A}$ are mutually orthogonal projections,
and, further,
\[
a=\sum_{j=1}^m\lambda_j e_j,
\quad
b=\sum_{k=1}^n\mu_k f_k,
\]
where $\lambda_1,\ldots,\lambda_m,\mu_1,\ldots,\mu_n\in[0,\infty[$ and $e_1,\ldots,e_m,f_1,\ldots,f_n\in\mathcal{A}$
are defined by
\begin{align*}
\lambda_1 & = \alpha_1,&
\mu_1     & = \beta_1,\\
\lambda_j & = \alpha_j-\alpha_{j-1} \ (1<j\le m),& 
\mu_k     & = \beta_k-\beta_{k-1} \ (1<k\le n),\\
e_j       & = p_j+\cdots+p_m \ (1\le j\le m),&
f_k       & = q_k+\cdots+q_n \ (1\le k\le n),
\end{align*}
as in the proof of Theorem \ref{t1945}. Since $e_jf_k=0$ for all $j\in\{1,\ldots,m\}$
and $k\in\{1,\ldots,n\}$, it follows that
\begin{equation*}
\begin{split}
\Vert aTb\Vert & =
\left\Vert\sum_{j=1}^m\sum_{k=1}^n\lambda_j\mu_k e_jTf_k\right\Vert\le
\sum_{j=1}^m\sum_{k=1}^n\lambda_j\mu_k \left\Vert e_jTf_k\right \Vert \\
& \le \sum_{j=1}^m\sum_{k=1}^n\lambda_j\mu_k\beta(T)=
\Vert a\Vert\Vert b\Vert\beta(T)\le\beta(T),
\end{split}
\end{equation*}
as required.

Now consider the general case.
Since $\mathcal{A}$ has real rank zero, it follows that
there exists a sequence $(c_n)$ in $\mathcal{A}_{\textup{sa}}$ such that
each $c_n$ has finite spectrum and $(c_n)$ converges to $a-b$ in norm.
For each $n\in\mathbb{N}$,
set 
\begin{equation*}
a_n=\tfrac{1}{2}\left(\vert c_n\vert+c_n \right),
\quad
b_n=\tfrac{1}{2} \left(\vert c_n\vert-c_n \right).
\end{equation*}
Since $(\vert c_n\vert)\to a+b$, we see that
$(a_n)\to a$ and $(b_n)\to b$.
Further, 
for each $n\in\mathbb{N}$, 
$a_n$, $b_n\in\mathcal{A}_+$, have finite spectra, and $a_n b_n=0$.
From the previous step we deduce that 
\[
\Vert a_n Tb_n\Vert\le\beta(T)\Vert a_n\Vert\Vert b_n\Vert
\quad
(n\in\mathbb{N}),
\] 
and so, taking limits on both sides of the above inequality, we obtain $\Vert aTb\Vert\le\beta(T)$,
which completes the proof.
\end{proof}

Now let $\mathcal{A}$ be a $C^*$-algebra, and let $\mathcal{X}$ a Banach right $\mathcal{A}$-module.
Then $\mathcal{X}^*$ is a Banach left $\mathcal{A}$-module with respect to the module operation specified by
\begin{equation*}
\langle x,a\phi\rangle=
\langle xa,\phi\rangle
\quad(\phi\in\mathcal{X}^*, \, a\in\mathcal{A}, \, x\in\mathcal{X}).
\end{equation*}
This module has the property that the map
$\phi\mapsto a\phi$ from $\mathcal{X}^*$ to $\mathcal{X}^*$ 
is weak* continuous for each $a\in\mathcal{A}$. 
Similarly, 
if $\mathcal{X}$ is a Banach left $\mathcal{A}$-module, then
$\mathcal{X}^*$ is a Banach right $\mathcal{A}$-module with respect to the module operation specified by
\begin{equation*} 
\langle x,\phi a\rangle=
\langle ax,\phi\rangle 
\quad
(\phi\in\mathcal{X}^*, \, a\in\mathcal{A}, \, x\in\mathcal{X}),
\end{equation*}
and
the map $\phi\mapsto \phi a$ from $\mathcal{X}^*$ to $\mathcal{X}^*$
is weak* continuous for each $a\in\mathcal{A}$.

\begin{theorem}\label{v1636}
Let $\mathcal{A}$ be a $C^*$-algebra, 
let $\mathcal{X}$ and $\mathcal{Y}$ be Banach right $\mathcal{A}$-modules
with $\mathcal{X}$ essential, and
let $T\colon\mathcal{X}\to\mathcal{Y}$ be a continuous linear map.
\begin{enumerate}
\item[(i)]
Suppose that $\left\{y\in\mathcal{Y} : y\mathcal{A}=0\right\}=\{0\}$ 
and that
\begin{equation*}
a,b\in\mathcal{A}_+, \  ab=0 \
\Longrightarrow \  aTb=0.
\end{equation*}
Then $T$ is a right $\mathcal{A}$-module homomorphism.
\item[(ii)]
Suppose that 
$\mathcal{Y}$ satisfies the condition
\[
\Vert y\Vert=\sup\left\{\Vert ya\Vert : a\in\mathcal{A},\ \Vert a\Vert=1\right\}
\]
for each $y\in\mathcal{Y}$.
Then 
\begin{equation*}
\Vert\ad (T)\Vert\le 8
\sup\bigl\{\Vert aTb\Vert: a,b\in\mathcal{A}_+ \text{ contractions, }  ab=0\bigr\}.
\end{equation*}
Further, the above condition holds in each of the following cases:
\begin{enumerate}
	\item[(a)]
	$\mathcal{Y}$ is essential;
	\item[(b)]
	$\mathcal{Y}$ is the dual of an essential Banach left $\mathcal{A}$-module.
\end{enumerate}
\end{enumerate}
\end{theorem}

\begin{proof}
Define the continuous bilinear map $\varphi\colon\mathcal{A}\times\mathcal{A}\to B(\mathcal{X},\mathcal{Y})$ by
\[
\varphi(a,b)=aTb
\quad
(a,b\in\mathcal{A}),
\]
and set $\varepsilon=\sup\bigl\{\Vert aTb\Vert: a,b\in\mathcal{A}_+\text{ contractions, } ab=0\bigr\}$.
Then $\Vert\varphi(a,b)\Vert\le\varepsilon\Vert a\Vert\Vert b\Vert$
whenever $a,b\in\mathcal{A}_+$ are such that $ab=0$.

Let $(e_i)_{i\in I}$ be an approximate identity for $\mathcal{A}$ of bound one.
As in the proof of Corollary~\ref{t1947}, we see that $(e_i)_{i\in I}$ has a subnet
$(e_j)_{j\in J}$ which converges to $1_{\mathcal{A}^{**}}$ in $\mathcal{A}^{**}$ with
respect to the weak* topology.

From Theorem~\ref{t1946} it follows that, for each $a\in\mathcal{A}$, the nets 
$(\varphi(a,e_j))_{j\in J}$ and $(\varphi(e_j,a))_{j\in J}$ converge in 
$B(\mathcal{X},\mathcal{Y})^{**}$ with respect to the weak* topology and
\begin{equation*}
\bigl\Vert\lim_{j\in J}\varphi(a,e_j)-\lim_{j\in J}\varphi(e_j,a)\bigr\Vert\le 
8\varepsilon \Vert a\Vert,
\end{equation*}
whence
\begin{equation*}
\bigl\Vert\lim_{j\in J}
aTe_j-\lim_{j\in J}e_jTa\bigr\Vert\le 
8\varepsilon \Vert a\Vert.
\end{equation*}
In particular, for each 
$x\in\mathcal{X}$,
$a,b\in\mathcal{A}$, and 
$\phi\in\mathcal{Y}^*$, 
the net 
\[
\bigl(\langle b\phi, (aTe_j-e_jTa)(x)\rangle\bigr)_{j\in J}=
\bigl(\langle \phi, T(xa)e_jb-T(xe_j)ab\rangle\bigr)_{j\in J}
\] 
converges and
\begin{equation}\label{ea922}
\left\vert\lim_{j\in J}\langle \phi, T(xa)e_jb-T(xe_j)ab\rangle\right\vert
\le8\varepsilon\Vert a\Vert\Vert b\Vert\Vert x\Vert\Vert\phi\Vert.
\end{equation}
Since $\mathcal{X}$ is essential, it follows that $(xe_j)_{j\in J}\to x$ in norm for each $x\in\mathcal{X}$. 
Thus \eqref{ea922} gives
\begin{equation*}
\bigl\vert
\langle\phi, T(xa)b-T(x)ab\rangle
\bigr\vert\le
8\varepsilon\Vert a\Vert\Vert b\Vert\Vert x\Vert\Vert\phi\Vert
\end{equation*}
for all
$x\in\mathcal{X}$,
$a,b\in\mathcal{A}$, and 
$\phi\in\mathcal{Y}^*$, 
and hence
\begin{equation}\label{ea923}
\Vert (T(xa)-T(x)a)b\Vert\le
8\varepsilon\Vert a\Vert\Vert b\Vert\Vert x\Vert
\end{equation}
for all
$x\in\mathcal{X}$, and
$a,b\in\mathcal{A}$.

(i) 
In this case, we have $\varepsilon=0$ and,
for each $x\in\mathcal{X}$ and $a\in\mathcal{A}$, 
\eqref{ea923} gives
\[
\bigl(T(xa)-T(x)a\bigr)b=0
\quad(b\in\mathcal{A}),
\]
which yields $T(xa)=T(x)a$.
Hence $T$ is a right $\mathcal{A}$-module homomorphism.

(ii)
In this case, for each $x\in\mathcal{X}$ and $a\in\mathcal{A}$, 
\eqref{ea923} gives
\[
\Vert T(xa)-T(x)a\Vert\le 8\varepsilon\Vert a\Vert\Vert x\Vert,
\]
so that $\Vert\ad (T)\Vert\le 8\varepsilon$, as claimed.

Now suppose that $\mathcal{Y}$ satisfies either of the additional assumptions (a) or (b);
we will prove that
$\Vert y\Vert=\sup\left\{\Vert ya\Vert : a\in\mathcal{A},\ \Vert a\Vert=1\right\}$ for each $y\in\mathcal{Y}$.
Take $y\in\mathcal{Y}$, and
set $\alpha=\sup\left\{\Vert ya\Vert: a\in\mathcal{A}, \ \Vert a\Vert=1\right\}$.
It is clear that $\alpha\le \Vert y\Vert$.

In case (a),
since $(ye_j)_{j\in J}\to y$, it follows that $(\Vert ye_j\Vert)_{j\in J}\to\Vert y\Vert$,
and consequently $\Vert y\Vert\le\alpha$.

In case (b),
$\mathcal{Y}$ is the dual of an essential Banach left $\mathcal{A}$-module $\mathcal{Y}_*$.
Take $\varepsilon>0$, and let
$\phi\in\mathcal{Y}_*$ with $\Vert \phi\Vert=1$ and
$\Vert y\Vert-\varepsilon<\left\vert\langle \phi,y\rangle\right\vert$.
Then $(e_j\phi)_{j\in J}\to \phi$ and the continuity of $\phi$ gives
$(\langle \phi,ye_j\rangle)_{j\in J}=
(\langle e_j\phi,y\rangle)_{j\in J}\to \langle \phi,y\rangle$,
which implies that there exists $j\in J$ such that
$\Vert y\Vert-\varepsilon<\vert\langle \phi, ye_j\rangle\vert\le\alpha$.
We thus get $\Vert y\Vert\le\alpha+\varepsilon$ for each $\varepsilon>0$,
and hence $\Vert y\Vert\le\alpha$, which completes the proof.
\end{proof}

\begin{theorem}\label{v1637}
Let $\mathcal{A}$ be a unital $C^*$-algebra of real rank zero,
let $\mathcal{X}$ and $\mathcal{Y}$ be unital quasi-Banach right $\mathcal{A}$-modules, and
let $T\colon\mathcal{X}\to\mathcal{Y}$ be a continuous linear map.
\begin{enumerate}
\item[(i)]
Suppose that
\begin{equation*}
e\in\mathcal{A} \text{ projection }\Longrightarrow \ e^\perp Te=0.
\end{equation*}
Then $T$ is a right $\mathcal{A}$-module homomorphism.
\item[(ii)]
Suppose that $\mathcal{X}$ and $\mathcal{Y}$ are Banach modules. 
Then
\begin{equation*}
\Vert\ad (T)\Vert\le 8
\sup\bigl\{\Vert e^\perp Te\Vert : e\in\mathcal{A}\text{ projection}\bigr\}.
\end{equation*}
\end{enumerate}
\end{theorem}

\begin{proof}
Define the continuous bilinear map $\varphi\colon\mathcal{A}\times\mathcal{A}\to B(\mathcal{X},\mathcal{Y})$ by
\[
\varphi(a,b)=aTb
\quad
(a,b\in\mathcal{A}).
\]

(i)
In this case, $\varphi(e,e^\perp)=0$ for each projection $e\in\mathcal{A}$. 
Consequently, 
Theorem~\ref{t1945}(i) shows that $\varphi(a,1_\mathcal{A})=\varphi(1_\mathcal{A},a)$
for each $a\in\mathcal{A}$, which gives $aT=Ta$, and 
this is precisely the assertion of (i).

(ii)
Set $\varepsilon=\sup\bigl\{\Vert e^\perp Te\Vert : e\in\mathcal{A}\text{ projection}\bigr\}$.
Then $\Vert\varphi(e,e^\perp)\Vert\le\varepsilon$ for each projection $e\in\mathcal{A}$,
and Theorem~\ref{t1945}(ii) now shows that 
\[
\Vert aT-Ta\Vert=
\Vert\varphi(a,1_\mathcal{A})-\varphi(1_\mathcal{A},a)\Vert
\le 8\varepsilon\Vert a\Vert
\] 
for each $a\in\mathcal{A}$. We thus get $\Vert\ad (T)\Vert\le 8\varepsilon$, as claimed.
\end{proof}

\begin{theorem}\label{t43}
Let $\mathcal{A}$ be a $C^*$-algebra of real rank zero,
let $\mathcal{X}$ and $\mathcal{Y}$ be Banach right $\mathcal{A}$-modules with $\mathcal{X}$ essential, and
let $T\colon\mathcal{X}\to\mathcal{Y}$ be a continuous linear map.
\begin{enumerate}
\item[(i)]
Suppose that $\left\{y\in\mathcal{Y} : y\mathcal{A}=0\right\}=\{0\}$ and that
\begin{equation*}
e,f\in\mathcal{A}\text{ projections, } ef=0 \ \Longrightarrow \ eTf=0.
\end{equation*} 
Then $T$ is a right $\mathcal{A}$-module homomorphism.
\item[(ii)]
Suppose that 
$\mathcal{Y}$ satisfies the condition
\[
\Vert y\Vert=\sup\left\{\Vert ya\Vert : a\in\mathcal{A},\ \Vert a\Vert=1\right\}
\]
for each $y\in\mathcal{Y}$. Then 
\begin{equation*}
\Vert\ad (T)\Vert\le 8
\sup\bigl\{
\Vert eTf\Vert : e,f\in\mathcal{A} \text{ projections, } ef=0\bigr\}.
\end{equation*}
Further, the above condition holds in each of the following cases:
\begin{enumerate}
\item[(a)]
$\mathcal{Y}$ is essential;
\item[(b)]
$\mathcal{Y}$ is the dual of an essential Banach left $\mathcal{A}$-module.
\end{enumerate}
\end{enumerate}
\end{theorem}

\begin{proof}
Take an approximate identity $(e_j)_{j\in J}$ for $\mathcal{A}$ consisting of projections.
Fix $j\in J$,  define the continuous bilinear map 
$\varphi_j\colon e_j\mathcal{A}e_j\times e_j\mathcal{A}e_j\to B(\mathcal{X},\mathcal{Y})$
by
\[
\varphi_j(a,b)=aTb
\quad
(a,b\in e_j\mathcal{A}e_j),
\]
and set $\varepsilon=\sup\bigl\{\Vert eTf\Vert: e,f\in\mathcal{A}\text{ projections, } ef=0\bigr\}$.
Then $e_j\mathcal{A}e_j$ is a unital $C^*$-algebra (with unit $e_j$) and has real rank zero.
Further, $\Vert\varphi(e,e_j-e)\Vert\le\varepsilon$ for each projection $e\in e_j\mathcal{A}e_j$.
From Theorem~\ref{t1945}(ii) it follows that
\[
\Vert aTe_j-e_jTa\Vert=
\Vert\varphi(a,e_j)-\varphi(e_j,a)\Vert
\le 8\varepsilon\Vert a\Vert
\]
for each $a\in e_j\mathcal{A}e_j$.
Hence
\begin{equation}\label{e66}
\left\Vert T(xe_jae_j)e_jb-T(xe_j)e_jae_jb \right\Vert
\le 8\varepsilon\Vert x\Vert\Vert a\Vert\Vert b\Vert
\quad
(j\in J, \, x\in\mathcal{X}, \, a,b\in\mathcal{A}).
\end{equation}
For each $x\in\mathcal{X}$ and $a,b\in\mathcal{A}$, we have
\begin{itemize}
\item 
$(e_jae_j)_{j\in J}\to a$ and $(e_jb)_{j\in J}\to b$ in norm, so that
(using the continuity of $T$)
$(T(xe_jae_j)e_jb)_{j\in J}\to T(xa)b$ in norm;
\item 
$(xe_j)_{j\in J}\to x$ in norm, because $\mathcal{X}$ is essential,
and $(e_jae_jb)_{j\in J}\to ab$ in norm, and hence (using the continuity of $T$)
$(T(xe_j)e_jae_jb)_{j\in J}\to T(x)ab$ in norm.
\end{itemize}
Thus, 
taking the limit in \eqref{e66} we see that
\[
\Vert T(xa)b-T(x)ab\Vert\le 8\varepsilon\Vert x\Vert\Vert a\Vert \Vert b\Vert
\quad
(x\in\mathcal{X}, \, a,b\in\mathcal{A}).
\]
The rest of the proof goes through as for Theorem~\ref{v1636} 
(from inequality \eqref{ea923} on).
\end{proof}

\begin{corollary}\label{t42}
Let $\mathcal{A}$ be a $C^*$-algebra, and
let $\mathcal{X}$ and $\mathcal{Y}$ be Banach right $\mathcal{A}$-modules.
Suppose that $\mathcal{X}$ is essential and that
$\{y\in\mathcal{Y} : y\mathcal{A}=0\}=\{0\}$.
Then the space $\Hom_\mathcal{A}(\mathcal{X},\mathcal{Y})$ is reflexive.
\end{corollary}

\begin{proof}
Take
$T\in B(\mathcal{X},\mathcal{Y})$ such that
\[
T(x)\in\overline{
\left\{\Phi(x) : \Phi\in\Hom_\mathcal{A}(\mathcal{X},\mathcal{Y}) 
\right\}
}
\quad
(x\in\mathcal{X}).
\]
Let $a,b\in\mathcal{A}$ be such that $ab=0$. We claim that $aTb=0$.
For each $x\in\mathcal{X}$, there exists a sequence $(\Phi_n)$ in 
$\Hom_\mathcal{A}(\mathcal{X},\mathcal{Y})$ such that
$(\Phi_n(xa))\to T(xa)$ in norm, and hence
\[
(aTb)(x)=
T(xa)b=\lim_{n\to\infty}\Phi_n(xa)b=
\lim_{n\to\infty}\Phi_n(xab)=0,
\]
which proves our claim.

From Theorem~\ref{v1636}(i), it follows that 
$T\in \Hom_\mathcal{A}(\mathcal{X},\mathcal{Y})$.
\end{proof}

\begin{corollary}\label{t44}
Let $\mathcal{A}$ be a unital $C^*$-algebra of real rank zero, and
let $\mathcal{X}$ and $\mathcal{Y}$ be unital quasi-Banach right $\mathcal{A}$-modules.
Then the space $\Hom_\mathcal{A}(\mathcal{X},\mathcal{Y})$ is reflexive.
\end{corollary}

\begin{proof}
This follows by the same method as in Corollary~\ref{t42}, with Theorem~\ref{v1636}(i) replaced by Theorem~\ref{v1637}(i).
\end{proof}

\subsection{Homomorphisms between non-commutative $L^p$-spaces}

Let $\mathcal{M}$ be a von Neumann algebra.
For each $0<p\le\infty$, the space $L^p(\mathcal{M})$ is a contractive Banach $\mathcal{M}$-bimodule
or a contractive $p$-Banach $\mathcal{M}$-bimodule according to $p\geq 1$ or $p<1$.
More generally, if $0<p, q, r\le\infty$ are such that 
$\tfrac{1}{p}+\tfrac{1}{q}=\tfrac{1}{r}$
(we adopt throughout the convention that $\tfrac{1}{\infty}=0$), then  
\begin{equation*}
x\in L^p(\mathcal{M}), \
y\in L^q(\mathcal{M})  \
\Longrightarrow \
xy\in L^r(\mathcal{M}) \text{ and }
\Vert xy\Vert_r\le\Vert x\Vert_p\Vert y\Vert_q.
\end{equation*}
This is the non-commutative H\"older's inequality.
From now on we confine attention to
the right $\mathcal{M}$-module structure of $L^p(\mathcal{M})$. 

\begin{theorem}\label{cont1}
Let $\mathcal{M}$ be a von Neumann algebra,
let $0<p,q\le \infty$, and
let $T\colon L^p(\mathcal{M})\to L^q(\mathcal{M})$ be a linear map.
Suppose that the map $e^\perp Te\colon L^p(\mathcal{M})\to L^q(\mathcal{M})$
is continuous for each projection $e\in\mathcal{M}$.
Then $T$ is continuous.
\end{theorem}

\begin{proof}
We first observe that
$eT-Te=eTe^\perp-e^\perp Te$ is continuous
for each projection $e\in\mathcal{M}$.
	
Now we consider the separating space of $T$, which is defined by 
\[
\mathcal{S}(T)=
\left\{y\in L^q(\mathcal{M}) :
\text{there exists } (x_n)\to 0\text{ in }L^p(\mathcal{M})\text{ with }(T(x_n))\to y\right\}.
\]
It is an immediate restatement of the closed graph theorem that $T$ is continuous 
if and only if $\mathcal{S}(T)=0$.
	
We claim that $\mathcal{S}(T)$ is a closed right submodule of $L^q(\mathcal{M})$.
By \cite[Proposition~5.1.2]{D}, $\mathcal{S}(T)$ is a closed subspace of $L^q(\mathcal{M})$.
Let $y\in\mathcal{S}(T)$, and let $e$ be a projection in $\mathcal{M}$.  
Take a sequence $(x_n)$ in $L^p(\mathcal{M})$ with $\lim x_n=0$ and $\lim T(x_n)=y$.
Then $\lim x_ne=0$ and, using the first observation,
\[
T(x_ne)=(eT-Te)(x_n)+T(x_n)e\to ye.
\]
Thus $ye\in\mathcal{S}(T)$.
Now let $a\in\mathcal{M}$ be an arbitrary element. Then there exists
a sequence $(a_n)$ in $\mathcal{M}$ such that each $a_n$ is a linear combination
of projections and $\lim a_n=a$. 
Since $\mathcal{S}(T)$ is a closed subspace of $L^q(\mathcal{M})$, it follows that
$ya_n\in\mathcal{S}(T)$ $(n\in\mathbb{N})$ and that
$ya=\lim ya_n\in\mathcal{S}(T)$.
Hence $\mathcal{S}(T)$ is a right submodule of $L^q(\mathcal{M})$, as claimed.
	
We now consider the subspace $\mathcal{I}(T)$ defined by
\[
\mathcal{I}(T)=\left\{a\in\mathcal{M} : \mathcal{S}(T)a=0 \right\}.
\]
It is clear that $\mathcal{I}(T)$ is a closed right ideal of $\mathcal{M}$.
Further, since $\mathcal{S}(T)$ is a right submodule of $L^q(\mathcal{M})$,
it follows immediately that $\mathcal{I}(T)$ is also a left ideal of $\mathcal{M}$.
Our next goal is to prove that $\mathcal{I}(T)$ is weak* closed in $\mathcal{M}$.
Take $y\in L^q(\mathcal{M})$, and let $s_r(y)$ be the right support projection of $y$.
Then
\[
\left\{a\in\mathcal{M} : ya=0 \right\} =
\left\{a\in\mathcal{M} : s_r(y)a=0\right\}
\]
(see \cite[Fact 1.2(ii)]{RX}) and, since $s_r(y)\in\mathcal{M}$, this latter set is clearly weak* closed in $\mathcal{M}$.
Since
\[
\mathcal{I}(T)=\bigcap_{y\in\mathcal{S}(T)} \left\{a\in\mathcal{M} : ya=0 \right\},
\]
we conclude that $\mathcal{I}(T)$ is weak* closed.

Since $\mathcal{I}(T)$ is a weak* closed two-sided ideal of $\mathcal{M}$,
it follows that there exists a central projection $e\in\mathcal{M}$ such
that 
\begin{equation*}
\mathcal{I}(T)=e\mathcal{M}.
\end{equation*}
We now claim that 
$\dim e^\perp\mathcal{M}<\infty$.
Assume towards a contradiction that $\dim e^\perp\mathcal{M}=\infty$.
Then we can take a sequence $(e_n)$ of non-zero mutually orthogonal projections 
in  $e^\perp\mathcal{M}$.
For $n\in\mathbb{N}$, we define the projection $f_n\in\mathcal{M}$ by
\[
f_n=\bigvee_{k\ge n} e_k,
\]
and consider the maps $R_n\in B(L^p(\mathcal{M}),L^p(\mathcal{M}))$
and $S_n\in B(L^q(\mathcal{M}), L^q(\mathcal{M}))$ defined by
\[
R_n(x)=x f_n, \quad
S_n(y)=y f_n  \quad
\left(x\in L^p(\mathcal{M}), \ y\in L^q(\mathcal{M}) \right).
\]
Our next objective is to apply a fundamental result about the separating space:
the so-called stability lemma.
By hypothesis, $TR_n-S_nT$ is continuous for each $n\in\mathbb{N}$, and hence,
by \cite[Corollary~5.2.7]{D},
$\bigl(\overline{S_1\cdots S_n(\mathcal{S}(T))}\bigr)$
is a nest in $L^q(\mathcal{M})$ which stabilizes. Since $f_{n+1}\le f_n$ for each $n\in\mathbb{N}$,
it follows that $S_1\cdots S_n=S_n$, and hence  that
\[
\overline{S_1\cdots S_n(\mathcal{S}(T))}=
\overline{\mathcal{S}(T)f_n}=
\mathcal{S}(T)f_n
\]
for each $n\in\mathbb{N}$.
Thus there exists $N\in\mathbb{N}$ such that
\begin{equation*}
\mathcal{S}(T)f_N=\mathcal{S}(T)f_n
\quad
(N\le n).
\end{equation*}
In particular, since $f_Ne_N=e_N$ and $f_{N+1}e_N=0$, we have
\begin{equation*}
\mathcal{S}(T)e_N=
\bigl(\mathcal{S}(T)f_N\bigr)e_N=
\bigl(\mathcal{S}(T)f_{N+1}\bigr)e_N=0.
\end{equation*}
Hence $e_N\in\mathcal{I}(T)=e\mathcal{M}$. 
But this is a contradiction of the facts that 
$e_N\in e^\perp\mathcal{M}$ and $e_N\ne 0$.

Our next claim is that the map 
$Te^\perp\colon L^p(\mathcal{M})\to L^q(\mathcal{M})$
is continuous. 
Since the projection $e^\perp$ is central, we see that 
$e^\perp x=xe^\perp$ for each $x\in L^p(\mathcal{M})$,
and hence
$e^\perp L^p(\mathcal{M})=e^\perp L^p(\mathcal{M})e^\perp$.
Moreover, \cite[Fact~1.4]{RX} shows that the subspace
$e^\perp L^p(\mathcal{M}) e^\perp$ is isometrically isomorphic to
$L^p(e^\perp\mathcal{M}e^\perp)$.
Since $\dim e^\perp\mathcal{M}<\infty$, it follows that
$\dim L^p(e^\perp\mathcal{M}e^\perp)<\infty$, so that
$\dim e^\perp L^p(\mathcal{M})<\infty$.
Thus the restriction of $T$ to the subspace $e^\perp L^p(\mathcal{M})$
is continuous, and hence the map 
\[
e^\perp T\colon  L^p(\mathcal{M})\to L^q(\mathcal{M}), \ x\mapsto T(x e^\perp)
\]
is continuous.
On the other hand, 
\[
Te^\perp=e^\perp T-(e^\perp T-Te^\perp),
\]
which implies  that $T e^\perp$ is continuous, as claimed.

Finally, we are in a position to prove the continuity of $T$.
From the above claim we deduce that 
$\mathcal{S}(T)e^\perp=0$, and hence that 
$e^\perp\in\mathcal{I}(T)=e\mathcal{M}$.
This implies that $e^\perp=0$, 
whence $1_\mathcal{M}=e\in\mathcal{I}(T)$, 
which gives $\mathcal{S}(T)=\mathcal{S}(T)1_\mathcal{M}=0$ and $T$ is continuous.
\end{proof}

\begin{corollary}
Let $\mathcal{M}$ be a von Neumann algebra, 
let $0< p,q\le\infty$, and
let $T\colon L^p(\mathcal{M})\to L^q(\mathcal{M})$ be a right $\mathcal{M}$-module homomorphism.
Then $T$ is continuous.
\end{corollary}

\begin{proof}
It is clear that $T$ satisfies the requirement in Theorem~\ref{cont1} and hence $T$ is continuous.
\end{proof}

Suppose that $0<p,q,r\le\infty$ are such that 
$\tfrac{1}{p}+\tfrac{1}{r}=\tfrac{1}{q}$.
By H\"older's inequality, for each $\xi\in L^r(\mathcal{M})$, 
we can define the left composition map $L_\xi\colon L^p(\mathcal{M})\to L^q(\mathcal{M})$ by
\[
L_\xi(x)=\xi x
\quad
(x\in L^p(\mathcal{M})).
\]
Further $L_\xi$ is continuous with $\Vert L_\xi\Vert\le\Vert \xi\Vert_r$,
and it is obvious that $L_\xi$ is a right $\mathcal{M}$-module homomorphism.

\begin{theorem}\label{js}
Let $\mathcal{M}$ be a von Neumann algebra, and
let $0< p,q\le\infty$.
\begin{enumerate}
\item[(i)]
Suppose that $p\ge q$, and let $0<r\le\infty$ be such that $\tfrac{1}{p}+\tfrac{1}{r}=\tfrac{1}{q}$.
Then the map 
\[
\xi\mapsto L_\xi, \ 
L^r(\mathcal{M})\to\Hom_\mathcal{M}\bigl(L^p(\mathcal{M}),L^q(\mathcal{M})\bigr)
\] 
is an isometric linear bijection.
\item[(ii)]
Suppose that $p<q$ and that $\mathcal{M}$ has no minimal projection. 
Then 
\[
\Hom_\mathcal{M}\bigl(L^p(\mathcal{M}),L^q(\mathcal{M})\bigr)=\left\{0\right\}.
\]
\end{enumerate}
\end{theorem}

\begin{proof}
(i)
By \cite[Theorem~2.5]{JS}, this map is a surjection.
We proceed to show that it is an isometry.
Let $\xi\in L^r(\mathcal{M})\setminus\{0\}$. 
We have already seen that $\Vert L_\xi\Vert\le \Vert\xi\Vert_r$.
We now establish the reverse inequality by considering three cases.

Assume that $p=\infty$.
Then $r=q$ and
\[
\Vert\xi\Vert_r=
\Vert L_\xi(1_\mathcal{M})\Vert_q\le
\Vert L_\xi\Vert\Vert 1_\mathcal{M}\Vert=
\Vert L_\xi\Vert,
\]
as required.

Now assume that $p<\infty$  and that $r=\infty$.
Then $p=q$, and, for each $x\in L^p(\mathcal{M})$, we have
\begin{equation*}
\begin{split}
\Vert L_\xi(x)\Vert_p & =
\Vert \xi x\Vert_p =
\bigl\Vert (\xi x)^*(\xi x)\bigr\Vert_{p/2}^{1/2}=
\bigl\Vert x^*\xi^*\xi x \bigr\Vert_{p/2}^{1/2} \\ 
& =
\bigl\Vert x^*\vert\xi\vert^2 x \bigr\Vert_{p/2}^{1/2}=
\bigl\Vert (\vert\xi\vert x)^*(\vert\xi\vert x)\bigr\Vert_{p/2}^{1/2}=
\bigl\Vert\vert\xi\vert x\bigr\Vert_p=
\bigl\Vert L_{\vert\xi\vert}(x) \bigr\Vert_p.
\end{split}
\end{equation*}
Thus $\Vert L_\xi\Vert=\Vert L_{\vert\xi\vert}\Vert$, and
\cite[Lemma~2.1]{JS} shows that $\Vert L_{\vert\xi\vert}\Vert=\Vert \vert\xi\vert\Vert=\Vert\xi\Vert$.

Finally, assume that $p,r<\infty$.
Then $\vert\xi\vert^{r/p}\in L^p(\mathcal{M})$, and we have
\begin{equation*}
\begin{split}
\bigl\Vert L_\xi\bigl(\vert\xi\vert^{r/p}\bigr)\bigr\Vert_q & =
\bigl\Vert\xi\vert\xi\vert^{r/p}\bigr\Vert_q=
\bigl\Vert\vert\xi\vert^{r/p}\xi^*\xi\vert\xi\vert^{r/p}\bigr\Vert_{q/2}^{1/2} \\ 
& =
\bigl\Vert\vert\xi\vert^{2(1+r/p)}\bigr\Vert_{q/2}^{1/2}=
\bigl\Vert\vert\xi\vert^{2r/q}\bigr\Vert_{q/2}^{1/2}=
\Vert\xi\Vert_r^{r/q}.
\end{split}
\end{equation*}
On the other hand, we have
\[
\bigl\Vert L_\xi\bigl(\vert\xi\vert^{r/p}\bigr)\bigr\Vert_q\le
\Vert L_\xi\Vert\bigl\Vert\vert\xi\vert^{r/p}\bigr\Vert_p=
\Vert L_\xi\Vert\Vert\xi\Vert_r^{r/p},
\]
and hence
\[
\Vert\xi\Vert_r=\Vert\xi\Vert_r^{r/q-r/p}\le\Vert L_\xi\Vert,
\]
as required.

(ii)
\cite[Corollary~2.7]{JS}.
\end{proof}

\begin{corollary}\label{c1549}
Let $\mathcal{M}$ be a von Neumann algebra, 
let $0< p,q\le\infty$, and
let $T\colon L^p(\mathcal{M})\to L^q(\mathcal{M})$ be a continuous linear map.
\begin{enumerate}
\item[(i)]
Suppose that
\[
e\in\mathcal{M}\text{ projection }
\Longrightarrow \
e^\perp Te=0.
\]
Then $T$ is a right $\mathcal{M}$-module homomorphism.
\item[(ii)]
Suppose that $1\le p,q\le\infty$.
Then 
\[
\Vert\ad (T)\Vert\le 8
\sup\bigl\{\Vert e^\perp Te\Vert : e\in\mathcal{M}\text{ projection}\bigr\}.
\]
\end{enumerate}
\end{corollary}

\begin{proof}
This follows from Theorem~\ref{v1637}.
\end{proof}

By Corollary~\ref{t44}, the space $\Hom_\mathcal{M}(L^p(\mathcal{M}),L^q(\mathcal{M}))$ is reflexive.
However we next show that this space is not merely reflexive.

\begin{corollary}
Let $\mathcal{M}$ be a von Neumann algebra,  let $0<p,q\le\infty$, and
let $T\colon L^p(\mathcal{M})\to L^q(\mathcal{M})$ be a linear map such that
\[
T(x)\in
\overline{\left\{\Phi(x) : \Phi\in\Hom_\mathcal{M}(L^p(\mathcal{M}),L^q(\mathcal{M}))\right\}}
\quad
(x\in L^p(\mathcal{M})).
\]
Then $T\in\Hom_\mathcal{M}(L^p(\mathcal{M}),L^q(\mathcal{M}))$.
In particular, the space
$\Hom_\mathcal{M}(L^p(\mathcal{M}),L^q(\mathcal{M}))$ is reflexive.
\end{corollary}

\begin{proof}
Take a projection $e\in\mathcal{M}$. 
For each $x\in L^p(\mathcal{M})$, there exists a sequence $(\Phi_n)$ in 
$\Hom_\mathcal{M}(L^p(\mathcal{M}),L^q(\mathcal{M}))$ such that
$T(xe^\perp)=\lim\Phi_n(xe^\perp)$ in norm, and hence
\[
(e^\perp Te)(x)=
T(xe^\perp)e=\lim_{n \to \infty}\Phi_n(xe^\perp)e=
\lim_{n \to \infty} \Phi_n(x)e^\perp e=0.
\]
This shows that $e^\perp Te=0$.

From Theorem~\ref{cont1}, it follows that $T$ is continuous, and  
Corollary~\ref{c1549} then gives
$T\in \Hom_\mathcal{M}(L^p(\mathcal{M}),L^q(\mathcal{M}))$.
\end{proof}

\section{Distance estimates}\label{s3}

\subsection{Homomorphisms between modules over a $C^*$-algebra}

Let $\mathcal{M}$ be a von Neumann algebra, and let $\mathcal{X}$ a Banach right $\mathcal{M}$-module.
Then the Banach left $\mathcal{M}$-module $\mathcal{X}^*$ is called normal if the map $a\mapsto a\phi$ 
from $\mathcal{M}$ to $\mathcal{X}^*$  is weak* continuous for each $\phi\in\mathcal{X}^*$.
Similarly, if $\mathcal{X}$ is a Banach left $\mathcal{M}$-module, 
then the Banach right $\mathcal{M}$-module $\mathcal{X}^*$ is called normal
if the map  $a\mapsto \phi a$ from $\mathcal{M}$ to $\mathcal{X}^*$ is weak* continuous for each $\phi\in\mathcal{X}^*$.

\begin{theorem}\label{at954}
Let $\mathcal{M}$ be an injective von Neumann algebra,
let $\mathcal{X}$ and $\mathcal{Y}$ be Banach right and left $\mathcal{M}$-modules, respectively, and
let $T\colon\mathcal{X}\to\mathcal{Y}^*$ be a continuous linear map.
\begin{enumerate}
\item[(i)] 
The subset $\mathcal{W}(T)$ of $\mathcal{X}$
consisting of the elements $x\in\mathcal{X}$
with the property that the bilinear map $(a,b)\mapsto T(xa)b$
from $\mathcal{M}\times\mathcal{M}$ to $\mathcal{Y}^*$ is separately weak* continuous 
is  a closed submodule of $\mathcal{X}$. Further,
if the modules $\mathcal{X}^*$ and $\mathcal{Y}^*$ are normal, 
then  $\mathcal{W}(T)=\mathcal{X}$.
\item[(ii)] 
There exists a continuous linear map $\Phi\colon\mathcal{X}\to\mathcal{Y}^*$ 
such that:
\begin{enumerate}
\item[(a)]
$\Vert\Phi\Vert\le \Vert T\Vert$;
\item[(b)]
$\Phi(xa)=\Phi(x)a$ for all $x\in\mathcal{W}(T)$ and $a\in\mathcal{M}$;
in particular, if both $\mathcal{X}^*$ and $\mathcal{Y}^*$ are normal, then $\Phi$ is a right $\mathcal{M}$-module homomorphism;
\item[(c)]
$\left\Vert 1_\mathcal{M}T-\Phi\right \Vert\le \Vert\ad (T)\Vert$;
in particular, if the module $\mathcal{X}$ is unital, then $\Vert T-\Phi\Vert\le \Vert\ad (T)\Vert$;
\item[(d)]
$\Vert T1_\mathcal{M}-\Phi\Vert\le \Vert\ad (T)\Vert$;
in particular, if the module $\mathcal{Y}$ is unital, then $\Vert T-\Phi\Vert\le \Vert\ad (T)\Vert$.
\end{enumerate}
\end{enumerate}	
\end{theorem}

\begin{proof}
(i) Routine verifications show that  $\mathcal{W}(T)$ is a submodule of $\mathcal{X}$.
To show that $\mathcal{W}(T)$ is closed, 
take a sequence $(x_n)$ in $\mathcal{W}(T)$ and $x\in\mathcal{X}$ such that $(x_n)\to x$ in norm. 
We define continuous bilinear maps
$\tau, \tau_n\colon \mathcal{M}\times\mathcal{M}\to\mathcal{Y}^*$ by
\[
\tau(a,b)=T(xa)b, \
\tau_n(a,b)=T(x_na)b \quad(a,b\in\mathcal{M}, \ n\in\mathbb{N}).
\]	
Then  $(\tau_n)\to\tau$ in norm, and each $\tau_n$ is separately weak* continuous. 
This implies that $\tau$ is separately weak* continuous, which shows that $x\in\mathcal{W}(T)$. 

Suppose that both $\mathcal{X}^*$ and $\mathcal{Y}^*$ are normal, and take $x\in\mathcal{X}$.
For each $a\in\mathcal{M}$, $T(xa)\in\mathcal{Y}^*$, so that, by definition, 
the map $b\mapsto T(xa)b$ from $\mathcal{M}$ to $\mathcal{Y}^*$ is weak*
continuous. For each $b\in\mathcal{M}$ and each $y\in\mathcal{Y}$, define
$\phi_{b,y}\in\mathcal{X}^*$ by
\[
\langle x,\phi_{b,y}\rangle=
\langle y,T(x)b\rangle
\quad
(x\in\mathcal{X}).
\]
Then
\[
\langle y, T(xa)b\rangle=
\langle x,a\phi_{b,y}\rangle
\quad 
(x\in\mathcal{X},\, a\in\mathcal{M}),
\]
and, since the map $a\mapsto a\phi_{b,y}$ is weak* continuous, it follows that the functional $a\mapsto\langle y, T(xa)b\rangle$ is weak* continuous for each
$x\in\mathcal{X}$. 

(ii)	
Let $G$ be  the discrete semigroup of the isometries of $\mathcal{M}$.
A mean on $G$ is a state $\mu$ on $\ell^\infty(G)$ and, for a given mean $\mu$,
we use the formal notation
\[
\int_G\phi(u) \, d\mu(u) :=\langle \phi,\mu\rangle
\quad
(\phi\in\ell^\infty(G)).
\]
By \cite[Theorem~2.1]{H}, there exists a mean $\mu$ on $G$ with the property that
\begin{equation}\label{e1038}
\int_G \tau(au^*,u)\, d\mu(u)=\int_G\tau(u^*,ua)\, d\mu(u)
\end{equation}
for each separately weak* continuous bilinear functional $\tau\colon\mathcal{M}\times\mathcal{M}\to\mathbb{C}$
and each $a\in\mathcal{M}$.

Define $\Phi\colon\mathcal{X}\to\mathcal{Y}^*$ by
\[
\langle y, \Phi(x)\rangle=\int_G \langle y,T(xu^*)u\rangle\, d\mu(u)
\quad 
(x\in\mathcal{X}, \, y\in\mathcal{Y}).
\]
Then $\Phi$ is well-defined and linear. 
Further, for each $x\in\mathcal{X}$, $y\in\mathcal{Y}$, and $u\in G$, we have
\begin{equation*}
\begin{split}
\vert \langle y,T(xu^*)u\rangle\vert & \le
\Vert y\Vert\Vert T(xu^*)u\Vert  \le \Vert y\Vert \Vert T(xu^*)\Vert \\ 
& \le \Vert y\Vert \Vert T\Vert\Vert xu^*\Vert  \le \Vert y\Vert \Vert T\Vert \Vert x\Vert,
\end{split}
\end{equation*}
which implies that
\[
\left\vert\int_G \langle y,T(xu^*)u\rangle\, d\mu(u)\right\vert\le
\Vert T\Vert \Vert x\Vert\Vert y\Vert
\]
and hence that
$\Vert\Phi(x)\Vert\le\Vert T\Vert \Vert x\Vert$.
Thus $\Phi$ is continuous and (a) holds.

Let $x\in\mathcal{W}(T)$, $a\in\mathcal{M}$, and $y\in\mathcal{Y}$. 
Since, by definition, the bilinear functional
$(u,v)\mapsto \langle y,T(xu)v\rangle$ is separately weak* continuous,
it follows from \eqref{e1038} that
\begin{equation*}
\begin{split}
\langle y,\Phi(xa)\rangle
&=
\int_G\langle y,T(xau^*)u\rangle\, d\mu(u)
=
\int_G\langle y,T(xu^*)ua\rangle\, d\mu(u)\\
&=
\int_G\langle ay,T(xu^*)u\rangle\, d\mu(u)
=
\langle ay,\Phi(x)\rangle =\langle y,\Phi(x)a\rangle.
\end{split}
\end{equation*}
This establishes (b).

Now let $x\in\mathcal{X}$ and $y\in\mathcal{Y}$ with 
$\Vert x\Vert=\Vert y\Vert=1$. Then
\begin{align*}
\vert\langle y,T(x1_\mathcal{M})-\Phi(x)\rangle\vert
& =
\left\vert
\int_G\langle y,T(x1_\mathcal{M})-T(xu^*)u\rangle \, d\mu(u)
\right\vert
\\ 
& = \left\vert
\int_G\langle y,T(xu^*u)-T(xu^*)u\rangle \, d\mu(u)
\right\vert \\
& \le
\int_G \bigl\vert \langle y,T(xu^*u)-T(xu^*)u)\rangle \bigr\vert \, d\mu(u) \\ 
& \le \int_G \left\Vert T(xu^*u)-T(xu^*)u \right\Vert \, d\mu(u)
\\
& \le
\int_G \Vert \ad (T)\Vert\Vert xu^*\Vert\Vert u\Vert \, d\mu(u)\\ 
& \le
\int_G\Vert\ad (T)\Vert\Vert x\Vert\Vert u^*\Vert\Vert u\Vert \, d\mu(u)
= \Vert\ad (T)\Vert.
\end{align*}
This gives (c).
We also have
\begin{align*}
\vert\langle y,T(x)1_\mathcal{M}-\Phi(x)\rangle\vert
&=
\left\vert
\int_G\langle y,T(x)1_\mathcal{M}-T(xu^*)u\rangle \, d\mu(u)
\right\vert
\\
&=
\left\vert
\int_G\langle y,T(x)u^*u-T(xu^*)u\rangle \, d\mu(u)
\right\vert
\\
&\le
\int_G \bigl\vert \langle y,T(x)u^*u-T(xu^*)u\rangle \bigr\vert \, d\mu(u)
\\
&\le
\int_G\Vert T(x)u^*u-T(xu^*)u\Vert \, d\mu(u)
\\
&\le
\int_G \Vert T(x)u^*-T(xu^*)\Vert \, d\mu(u)
\\
&\le
\int_G \Vert\ad (T)\Vert\Vert x\Vert\Vert u^*\Vert \, d\mu(u)
=
\Vert\ad (T)\Vert,
\end{align*}
and this gives (d).
\end{proof}

\begin{theorem}\label{at1140}
Let $\mathcal{A}$ be a nuclear $C^*$-algebra,
let $\mathcal{X}$ and $\mathcal{Y}$ be  Banach right and left $\mathcal{A}$-modules, respectively, and
let $T\colon\mathcal{X}\to\mathcal{Y}^*$ be a continuous linear map.
Then
there exists a continuous right $\mathcal{A}$-module homomorphism $\Phi\colon\mathcal{X}\to\mathcal{Y}^*$ 
such that:
\begin{enumerate}
\item[(a)]
$\Vert\Phi\Vert\le \Vert T\Vert$;
\item[(b)] 
$\Vert aT-a\Phi\Vert
\le 
\Vert\ad (T)\Vert\Vert a\Vert$ $(a\in\mathcal{A})$; moreover,
if the module $\mathcal{X}$ is essential, then $\Vert T-\Phi\Vert\le \Vert\ad (T)\Vert$;
\item[(c)]
$\Vert Ta-\Phi a\Vert
\le \Vert\ad (T)\Vert\Vert a\Vert$ $(a\in\mathcal{A})$; moreover,
if the module $\mathcal{Y}$ is essential, then
$\Vert T-\Phi\Vert\le \Vert\ad (T)\Vert$.
\end{enumerate}	
\end{theorem}

\begin{proof}
Consider the projective tensor product $\mathcal{A}\widehat{\otimes}\mathcal{A}$, and
let $\pi\colon\mathcal{A}\widehat{\otimes}\mathcal{A}\to\mathcal{A}$ be the continuous
linear map defined through
\[
\pi(a\otimes b)=ab
\quad
(a,b\in\mathcal{A}).
\]
The Banach space $\mathcal{A}\widehat{\otimes}\mathcal{A}$ is a contractive Banach 
$\mathcal{A}$-bimodule with respect to the operations defined through
\[
(a\otimes b)c=a\otimes bc, \ c(a\otimes b)=ca\otimes b \quad
(a,b,c\in\mathcal{A}).
\]
By \cite[Theorem~3.1]{H}, there exists a virtual diagonal for $\mathcal{A}$ of norm
one. This is an element
$\text{M}\in(\mathcal{A}\widehat{\otimes}\mathcal{A})^{**}$ 
with $\Vert\text{M}\Vert=1$ such that,
for each $a\in\mathcal{A}$, we have
\begin{equation*}
a\text{M}=\text{M}a \quad \text{and} \quad \pi^{**}(\text{M})a=a.
\end{equation*}
Here, both $(\mathcal{A}\widehat{\otimes}\mathcal{A})^{**}$ and $\mathcal{A}^{**}$ 
are considered as dual $\mathcal{A}$-bimodules in the usual way.
For each continuous bilinear functional $\tau\colon\mathcal{A}\times\mathcal{A}\to\mathbb{C}$
there exists a unique element $\widehat{\tau}\in(\mathcal{A}\widehat{\otimes}\mathcal{A})^{*}$
such that
\[
\widehat{\tau}(a\otimes b)=\tau(a,b)\quad (a,b\in\mathcal{A}),
\]
and we use the formal notation
\[
\int_{\mathcal{A}\times\mathcal{A}}\tau(u,v) \, d\text{M}(u,v) := \langle \widehat{\tau},\text{M}\rangle.
\]
Using this notation, the defining properties of $\text{M}$ can be written as
\begin{equation}\label{e1339}
\int_{\mathcal{A}\times\mathcal{A}}\tau(au,v)\, d\text{M}(u,v)=
\int_{\mathcal{A}\times\mathcal{A}}\tau(u,va)\, d\text{M}(u,v)
\end{equation}
and
\begin{equation}\label{e1340}
\int_{\mathcal{A}\times\mathcal{A}}\langle auv,\phi\rangle\, d\text{M}(u,v)=\langle a,\phi\rangle
\end{equation}
for each continuous bilinear functional $\tau\colon\mathcal{A}\times\mathcal{A}\to\mathbb{C}$,
each $a\in\mathcal{A}$, and each $\phi\in\mathcal{A}^*$;
further, it will be helpful noting that 
\begin{equation}\label{e1341}
\left\vert
\int_{\mathcal{A}\times\mathcal{A}}\tau(u,v)\, d\text{M}(u,v)
\right\vert
\le\Vert\text{M}\Vert\Vert\widehat{\tau}\Vert=\Vert\tau\Vert.
\end{equation}

Define $\Phi\colon\mathcal{X}\to\mathcal{Y}^*$ by
\[
\langle y,\Phi(x)\rangle=
\int_{\mathcal{A}\times\mathcal{A}}\langle y,T(xu)v\rangle\, d\text{M}(u,v)
\quad
(x\in\mathcal{X}, \, y\in\mathcal{Y}).
\]
Then $\Phi$ is well-defined and linear.
For each $x\in\mathcal{X}$, $y\in\mathcal{Y}$, and $u,v\in\mathcal{A}$, we have
\begin{equation*}
\begin{split}
\vert\langle y,T(xu)v\rangle\vert & \le
\Vert T(xu)v\Vert\Vert y\Vert\le
\Vert T(xu)\Vert\Vert v\Vert\Vert y\Vert \\
& \le \Vert T\Vert\Vert xu\Vert\Vert v\Vert\Vert y\Vert\le
\Vert T\Vert \Vert x\Vert\Vert u\Vert\Vert v\Vert\Vert y\Vert.
\end{split}
\end{equation*}
Then, using \eqref{e1341}, we have
\[
\left\vert
\int_{\mathcal{A}\times\mathcal{A}}\langle y,T(xu)v\rangle\, d\text{M}(u,v)
\right\vert\le
\Vert T\Vert\Vert x\Vert\Vert y\Vert,
\]
which implies that $\Vert\Phi(x)\Vert\le \Vert T\Vert\Vert x\Vert$.
Thus $\Phi$ is continuous and (a) holds.

We claim that $\Phi$ is a right $\mathcal{A}$-module homomorphism. Indeed,
for $x\in\mathcal{X}$, $a\in\mathcal{A}$, and each $y\in\mathcal{Y}$,  \eqref{e1339} gives
\begin{equation*}
\begin{split}
\langle y,\Phi(xa)\rangle
&=
\int_{\mathcal{A}\times\mathcal{A}}\langle y,T(xau)v\rangle \, d\text{M}(u,v)
=
\int_{\mathcal{A}\times\mathcal{A}}\langle y,T(xu)va\rangle \, d\text{M}(u,v)\\
&=
\int_{\mathcal{A}\times\mathcal{A}}\langle ay,T(xu)v\rangle\, d\text{M}(u,v)
 =
\langle ay,\Phi(x)\rangle =
\langle y,\Phi(x)a\rangle.
\end{split}
\end{equation*}
  
Our next objective is to prove (b).	
Take
$x\in\mathcal{X}$, $a\in\mathcal{A}$, and $y\in\mathcal{Y}$
with $\Vert x\Vert=\Vert a\Vert=\Vert y\Vert=1$, and define
$\phi\in\mathcal{A}^*$ and $\tau\colon\mathcal{A}\times\mathcal{A}\to\mathbb{C}$ by
\begin{gather*}
\langle u,\phi\rangle=
\langle y,T(xu)\rangle
\quad
(u\in\mathcal{A}), \\
\tau(u,v)=\left\langle y, T(xauv)-T(xau)v \right\rangle
\quad
(u,v\in\mathcal{A}).
\end{gather*}
For each $u,v\in\mathcal{A}$, we have
\begin{equation*}
\begin{split}
\vert \tau(u,v)\vert & \le
\Vert T(xauv)-T(xau)v\Vert\le
\Vert\ad (T)\Vert\Vert xau\Vert \Vert v\Vert\\ 
& \le 
\Vert\ad (T)\Vert\Vert x\Vert\Vert au\Vert\Vert v\Vert\le
\Vert\ad (T)\Vert\Vert u\Vert\Vert v\Vert,
\end{split}
\end{equation*}
so that $\Vert \tau\Vert\le \Vert\ad (T)\Vert$.
By \eqref{e1340},
\[
\langle y,T(xa)\rangle=
\langle a,\phi\rangle=
\int_{\mathcal{A}\times\mathcal{A}}\langle auv,\phi\rangle \, d\text{M}(u,v)=
\int_{\mathcal{A}\times\mathcal{A}}\langle y,T(xauv)\rangle \, d\text{M}(u,v),
\] 
and, using the definition of $\Phi$, we obtain
\[
\langle y, T(xa)-\Phi(xa)\rangle=
\int_{\mathcal{A}\times\mathcal{A}}\tau(u,v) \, d\text{M}(u,v).
\]
By \eqref{e1341},
$\vert\langle y, T(xa)-\Phi(xa)\rangle\vert\le
\Vert\ad (T)\Vert$.
Since this inequality holds for each $y\in\mathcal{Y}$ with $\Vert y\Vert=1$,
it follows that 
\[
\Vert T(xa)-\Phi(xa)\Vert\le \Vert\ad (T)\Vert.
\]
Now assume that $\mathcal{X}$ is essential.
Take an approximate identity $(e_j)_{j\in J}$ for $\mathcal{A}$ of bound $1$.
Then $(e_j)_{j\in J}$ is a right approximate identity for $\mathcal{X}$ and,
for each $x\in\mathcal{X}$ with $\Vert x\Vert=1$,
\[
\left\Vert T(xe_j)-\Phi(xe_j)\right\Vert\le \Vert\ad (T)\Vert
\quad
(j\in J),
\]
so that, using the continuity of $T$ and $\Phi$, we see that
$\Vert T(x)-\Phi(x)\Vert\le \Vert\ad (T)\Vert$. Thus
$\Vert T-\Phi\Vert \le \Vert\ad (T)\Vert$.

Finally, we proceed to prove (c).  
Take $x\in\mathcal{X}$, $a\in\mathcal{A}$, and $y\in\mathcal{Y}$
with $\Vert x\Vert=\Vert a\Vert=\Vert y\Vert=1$, and define
$\phi\in\mathcal{A}^*$ and $\tau\colon\mathcal{A}\times\mathcal{A}\to\mathbb{C}$ by
\begin{gather*}
\langle u,\phi\rangle=
\langle y,T(x)u\rangle
\quad
(u\in\mathcal{A}),
\\
\tau(u,v)=\langle y, T(x)auv-T(xau)v\rangle
\quad
(u,v\in\mathcal{A}).
\end{gather*}
For each $u,v\in\mathcal{A}$, we have
\begin{equation*}
\begin{split}
\vert \tau(u,v)\rangle\vert & \le
\Vert T(x)auv-T(xau)v\Vert\le
\Vert T(x)au-T(xau)\Vert\Vert v\Vert \\
& \le
\Vert\ad (T)\Vert\Vert x\Vert\Vert au\Vert\Vert v\Vert\le
\Vert\ad (T)\Vert\Vert u\Vert\Vert v\Vert,
\end{split}
\end{equation*}
so that $\Vert \tau\Vert\le \Vert\ad (T)\Vert$.
By \eqref{e1340},
\[
\langle y,T(x)a\rangle=
\langle a,\phi\rangle=
\int_{\mathcal{A}\times\mathcal{A}}\langle auv,\phi\rangle\, d\text{M}(u,v)=
\int_{\mathcal{A}\times\mathcal{A}}\langle y,T(x)auv\rangle\, d\text{M}(u,v),
\] 
and, using the definition of $\Phi$, we obtain
\[
\langle y, T(x)a-\Phi(x)a\rangle=
\langle y, T(x)a-\Phi(xa)\rangle=
\int_{\mathcal{A}\times\mathcal{A}}\tau(u,v) \, d\text{M}(u,v).
\]
From \eqref{e1341} we see that
$\vert\langle y, T(x)a-\Phi(x)a\rangle\vert\le
\Vert\ad (T)\Vert$.
Thus
\[
\Vert T(x)a-\Phi(x)a\Vert\le \Vert\ad (T)\Vert.
\]
Assume that $\mathcal{Y}$ is essential, and
take an approximate identity $(e_j)_{j\in J}$ for $\mathcal{A}$ of bound $1$.
For each $x\in\mathcal{X}$ and $y\in\mathcal{Y}$ with 
$\Vert x\Vert=\Vert y\Vert=1$, we have
\[
\vert\langle e_jy,T(x)-\Phi(x)\rangle\vert=
\vert\langle y, T(x)e_j-\Phi(x)e_j\rangle\vert\le \Vert\ad (T)\Vert
\quad
(j\in J).
\]
Since $\mathcal{Y}$ is essential, it follows that
$(e_j)_{j\in J}$ is a right approximate identity for $\mathcal{Y}$ and
hence, taking limit, we see that
$\vert\langle y,T(x)-\Phi(x)\rangle\vert\le\Vert\ad (T)\Vert$.
Therefore
$\Vert T(x)-\Phi(x)\Vert\le \Vert\ad (T)\Vert$,
and the proof is complete.
\end{proof}

\begin{corollary}\label{s1816}
Let $\mathcal{M}$ be an injective von Neumann algebra, and
let $\mathcal{X}$ and $\mathcal{Y}$ be unital Banach right and left $\mathcal{M}$-modules, respectively, 
with both  $\mathcal{X}^*$ and  $\mathcal{Y}^*$ normal.
Then
\begin{equation*}
\dist 
\bigl(T,\Hom_\mathcal{M}(\mathcal{X},\mathcal{Y}^*)\bigr)
\le 8\sup\bigl\{\Vert e^\perp Te\Vert : e\in\mathcal{M}\text{ projection}\bigr\}
\end{equation*}
for each $T\in B(\mathcal{X},\mathcal{Y}^*)$. In particular,
the space $\Hom_\mathcal{A}(\mathcal{X},\mathcal{Y}^*)$ is hyperreflexive.
\end{corollary}

\begin{proof}
Take  $T\in B(\mathcal{X},\mathcal{Y}^*)$.
Then Theorem~\ref{at954} gives $\Phi\in\text{Hom}_\mathcal{M}(\mathcal{X},\mathcal{Y}^*)$
such that $\Vert\Phi\Vert\le \Vert T\Vert$ and $\Vert T-\Phi\Vert\le\Vert\ad (T)\Vert$.
Theorem~\ref{v1637}(ii) now shows that
\[
\Vert T-\Phi\Vert\le
8\sup\bigl\{\Vert e^\perp Te\Vert : e\in\mathcal{M}\text{ projection}\bigr\},
\]
which establishes our estimate of the distance to $\text{Hom}_\mathcal{M}(\mathcal{X},\mathcal{Y}^*)$.

The hyperreflexivity follows from the estimates in Proposition~\ref{p10}.
\end{proof}

\begin{corollary}
Let $\mathcal{A}$ be a nuclear $C^*$-algebra, and
let $\mathcal{X}$ and $\mathcal{Y}$ be essential Banach right and left $\mathcal{A}$-modules, respectively.
Then
\begin{equation*}
\dist 
\bigl(T,\Hom_\mathcal{A}(\mathcal{X},\mathcal{Y}^*)\bigr)
\le 8\sup\bigl\{\Vert aTb\Vert: a,b\in\mathcal{A}_+\text{ contractions, } ab=0\bigr\}
\end{equation*}
for each $T\in B(\mathcal{X},\mathcal{Y}^*)$. In particular,
the space $\Hom_\mathcal{A}(\mathcal{X},\mathcal{Y}^*)$ is hyperreflexive.
\end{corollary}

\begin{proof}
The estimate follows from 
Theorem~\ref{v1636}(ii) and Theorem~\ref{at1140}, as in Corollary~\ref{s1816}.
The hyperreflexivity follows from the estimates in Proposition~\ref{p10}.
\end{proof}

\begin{corollary}
Let $\mathcal{A}$ be a unital nuclear $C^*$-algebra of real rank zero, and
let $\mathcal{X}$ and $\mathcal{Y}$ be unital Banach right and left $\mathcal{A}$-modules, respectively.
Then
\begin{equation*}
\dist
\bigl(T,\Hom_\mathcal{A}(\mathcal{X},\mathcal{Y}^*)\bigr)
\le 8\sup\bigl\{\Vert e^\perp Te\Vert : e\in\mathcal{A}\text{ projection}\bigr\}
\end{equation*}
for each $T\in B(\mathcal{X},\mathcal{Y}^*)$. 
\end{corollary}

\begin{proof}
The estimate follows from
Theorem~\ref{v1637}(ii) and Theorem~\ref{at1140}, as in Corollary~\ref{s1816}.
\end{proof}

\begin{corollary}
Let $\mathcal{A}$ be a nuclear $C^*$-algebra  of real rank zero, and
let $\mathcal{X}$ and $\mathcal{Y}$ be essential Banach right and left $\mathcal{A}$-modules, respectively.
Then
\begin{equation*}
\dist
\bigl(T,\Hom_\mathcal{A}(\mathcal{X},\mathcal{Y}^*)\bigr)
\le 8\sup\bigl\{\Vert eTf\Vert : e,f\in\mathcal{A} \text{ projections, } ef=0\bigr\}
\end{equation*}
for each $T\in B(\mathcal{X},\mathcal{Y}^*)$. 
\end{corollary}

\begin{proof}
The estimate follows from
Theorem~\ref{t43}(ii) and Theorem~\ref{at1140}, as in Corollary~\ref{s1816}.
\end{proof}

\subsection{Homomorphisms between non-commutative $L^p$-spaces}

Let $\mathcal{M}$ be a von Neumann algebra.
For each $1\le p\le\infty$, define $1\le p^*\le\infty$ by the requirement that $\tfrac{1}{p}+\tfrac{1}{p^*}=1$.
There exists a natural isomorphism $\omega\mapsto x_\omega$ from $\mathcal{M}_*$ onto
$L^1(\mathcal{M})$ (this isomorphism preserves the adjoint operation, positivity, and polar decomposition),
and hence the space $L^1(\mathcal{M})$ is equipped with a distinguished contractive positive linear functional
$\traza$ defined by $\traza(x_\omega)=\omega(1_\mathcal{M})$ $(\omega\in\mathcal{M}_*)$.
This functional implements, for each $1\le p\le\infty$, the duality 
$\langle\cdot,\cdot\rangle\colon L^p(\mathcal{M})\times L^{p^*}(\mathcal{M})\to\mathbb{C}$
defined by
\[
\langle x,y\rangle=\traza (xy)=\traza(yx)
\quad
(x\in L^p(\mathcal{M}), \, y\in L^{p^*}(\mathcal{M})).
\]
In the case where $p\ne\infty$,
the above duality gives an isometric isomorphism from $L^{p^*}(\mathcal{M})$ onto $L^p(\mathcal{M})^*$ . 
Moreover the duality satisfies the following properties:
\begin{align}
\langle ax,y\rangle & =\langle x,ya\rangle, \quad
\langle xa,y\rangle=\langle x,ay\rangle, \label{e227} \\
\langle ax,y\rangle & =\langle xy,a\rangle, \quad
\langle xa,y\rangle=\langle yx,a\rangle \label{e228}
\end{align}
for all $x\in L^p(\mathcal{M})$, $y\in L^{p^*}(\mathcal{M})$, and $a\in\mathcal{M}$.
Condition \eqref{e227} shows that, for $p\ne\infty$, 
the identification of $L^{p^*}(\mathcal{M})$ with $L^p(\mathcal{M})^*$
is an isomorphism of $\mathcal{M}$-bimodules, and, further, condition \eqref{e228} shows that
$L^p(\mathcal{M})^*$ is a normal $\mathcal{M}$-bimodule.

\begin{theorem}\label{ac1236}
Let $\mathcal{M}$ be a von Neumann algebra, and
let $1\le p,q\le\infty$.
Then
\begin{equation*}
\dist\bigl(T,\Hom_\mathcal{M}(L^\infty(\mathcal{M}),L^q(\mathcal{M}))\bigr)
\le
8\sup\bigl\{\Vert e^\perp Te\Vert : e\in\mathcal{M}\text{ projection}\bigr\}
\end{equation*}
for each $T\in B(L^\infty(\mathcal{M}),L^q(\mathcal{M}))$, and
\begin{equation*}
\dist\bigl(T,\Hom_\mathcal{M}(L^p(\mathcal{M}),L^1(\mathcal{M}))\bigr)
\le
8\sup\bigl\{\Vert e^\perp Te\Vert : e\in\mathcal{M}\text{ projection}\bigr\}
\end{equation*}
for each $T\in B(L^p(\mathcal{M}),L^1(\mathcal{M}))$.
In particular, the spaces $\Hom_\mathcal{M}(L^\infty(\mathcal{M}),L^q(\mathcal{M}))$ and 
$\Hom_\mathcal{M}(L^p(\mathcal{M}),L^1(\mathcal{M}))$
are hyperreflexive.
\end{theorem}

\begin{proof}
Suppose that $T\in B(L^\infty(\mathcal{M}),L^q(\mathcal{M}))$.
Define $\xi=T(1_\mathcal{M})\in L^q(\mathcal{M})$.
Then, 
for each $x\in\mathcal{M}$, we have
\[	
\Vert (T-L_\xi)(x)\Vert_q=
\Vert T(1_\mathcal{M}x)-T(1_\mathcal{M})x\Vert_q\le
\Vert\ad (T)\Vert\Vert 1_\mathcal{M}\Vert\Vert x\Vert,
\]
so that
$\Vert T-L_\xi\Vert\le\Vert\ad (T)\Vert$.
Corollary~\ref{c1549} now gives 
\[
\Vert T-L_\xi\Vert\le 8
\sup\bigl\{\Vert e^\perp Te\Vert : e\in\mathcal{M}\text{ projection}\bigr\},
\]
which establishes the required inequality.

Now suppose that $T\in B(L^p(\mathcal{M}),L^1(\mathcal{M}))$.	
In order to get the desired inequality,
we are reduced to consider the case $p\ne \infty$.
Consider the continuous linear functional $\phi$ on $L^p(\mathcal{M})$ defined by
\[
\langle x,\phi\rangle=\traza (T(x))
\quad
(x\in L^p(\mathcal{M})).
\]
Then there exists $\xi\in L^{p^*}(\mathcal{M})$ such that $\Vert\xi\Vert=\Vert\phi\Vert\le\Vert T\Vert$ and
\[
\traza (\xi x)=
\langle x,\phi\rangle=
\traza \bigl(T(x)\bigr)
\quad 
\bigl(x\in L^p(\mathcal{M})\bigr).
\]
For each $x\in L^p(\mathcal{M})$ and $a\in \mathcal{M}$, we see that
\begin{equation*}
\begin{split}
\bigl\vert\traza \bigl((T-L_\xi)(x)a\bigr)\bigr\vert & =
\bigl\vert\traza \bigl(T(x)a-\xi xa)\bigr)\bigr\vert=
\bigl\vert\traza \bigl(T(x)a-T(xa)\bigr)\bigr\vert \\
& \le \Vert T(x)a-T(xa)\Vert_1\le
\Vert\ad (T)\Vert\Vert x\Vert_p\Vert a\Vert.
\end{split}
\end{equation*}
This implies that
$\Vert (T-L_\xi)(x)\Vert_1\le\Vert\ad (T)\Vert\Vert x\Vert_p$,
whence $\Vert T-L_\xi\Vert\le\Vert\ad (T)\Vert$, and 
Corollary~\ref{c1549} shows that 
\[
\Vert T-L_\xi\Vert\le 8
\sup\bigl\{\Vert e^\perp Te\Vert : e\in\mathcal{M}\text{ projection}\bigr\}.
\]

The hyperreflexivity follows from the estimates in Proposition~\ref{p10}.
\end{proof}

\begin{theorem}\label{t1924}
Let $\mathcal{M}$ be an injective von Neumann algebra, and
let $1\le p,q\le \infty$. 
Then
\[
\dist
\bigl(T,\Hom_\mathcal{M}(L^p(\mathcal{M}),L^q(\mathcal{M})\bigr)\le
8
\sup\bigl\{\Vert e^\perp Te\Vert : e\in\mathcal{M}\text{ projection}\bigr\}
\]
for each $T\in B(L^p(\mathcal{M}),L^q(\mathcal{M}))$.
In particular, the space $\Hom_\mathcal{M}(L^p(\mathcal{M}),L^q(\mathcal{M}))$ is hyperreflexive
and the hyperreflexivity constant is at most $8$.
\end{theorem}

\begin{proof}
By Theorem~\ref{ac1236}, we need only to consider the case where $p\ne\infty$ and $q\ne 1$, 
and then the result follows from Corollary~\ref{s1816}, because both modules $L^p(\mathcal{M})^*$
and $L^q(\mathcal{M})\bigl(=L^{q^*}(\mathcal{M})^*\bigr)$ are normal.
\end{proof}

At the expense of replacing the condition $1\le p,q\le\infty$ by $1\le q<p\le\infty$ and losing
the bound $8$ on the distance estimate, we may remove the injectivity of the von Neumann algebra
$\mathcal{M}$ in Theorem \ref{t1924}.
To this end, we will be involved with the ultraproduct of non-commutative $L^p$-spaces.
We summarize some of its main properties.

Let $(\mathcal{X}_n)$ be a sequence of Banach spaces and let $\mathcal{U}$ be an ultrafilter on $\mathbb{N}$.
Let $\prod\mathcal{X}_n$ be the $\ell^\infty$-sum of the sequence $(\mathcal{X}_n)$ and take
\[
N_\mathcal{U}=\bigl\{(x_n)\in{\textstyle\prod}\mathcal{X}_n : 
\sideset{}{_\mathcal{U}}\lim\Vert x_n\Vert=0\bigr\}.
\] 
Then the ultraproduct $\prod_\mathcal{U}\mathcal{X}_n$ of the sequence $(\mathcal{X}_n)$
along $\mathcal{U}$ is the quotient Banach space $\prod\mathcal{X}_n/N_\mathcal{U}$.
Given $(x_n)\in\prod\mathcal{X}_n$, 
we write $(x_n)_\mathcal{U}$ for its corresponding equivalence class 
in $\prod_\mathcal{U}\mathcal{X}_n$. The norm on $\prod_\mathcal{U}\mathcal{X}_n$ is given by
\[
\bigl\Vert (x_n)_\mathcal{U}\bigr\Vert=\lim_\mathcal{U}\Vert x_n\Vert
\] 
for each $(x_n)_\mathcal{U}\in\prod_\mathcal{U}\mathcal{X}_n$. 
Let $(\mathcal{Y}_n)$ be another sequence of Banach spaces and let
$(T_n)\in\prod B(\mathcal{X}_n,\mathcal{Y}_n)$.
Then we define 
$\prod_\mathcal{U}T_n\colon\prod_\mathcal{U}\mathcal{X}_n\to\prod_\mathcal{U}\mathcal{Y}_n$
by
\[
{\textstyle\prod}_\mathcal{U} T_n\bigl((x_n)_\mathcal{U}\bigr)=
\bigl(T_n(x_n)\bigr)_\mathcal{U}
\]
for each $(x_n)_\mathcal{U}\in\prod_\mathcal{U}\mathcal{X}_n$.
Of course, it can be checked that the definition we make is independent of 
the choice of the representative of the equivalence class. Moreover,
$\prod_\mathcal{U}T_n$ is continuous and
\begin{equation}\label{e2027}
\bigl\Vert {\textstyle\prod}_\mathcal{U} T_n\bigr\Vert=\lim_\mathcal{U}\Vert T_n\Vert.
\end{equation}
All the above statements are also valid for quasi-Banach spaces.
We refer the reader to \cite{He} for the basics of ultraproducts.

If $(\mathcal{A}_n)$ is a sequence of $C^*$-algebras, then $\prod_\mathcal{U}\mathcal{A}_n$ is
again a $C^*$-algebra. The ultraproduct of a sequence $(\mathcal{M}_n)$ of von Neumann 
algebras is not as straightforward as the $C^*$-algebra case.
According to \cite{Gr,R}, it is known that $\prod_\mathcal{U}L^1(\mathcal{M}_n)$ is isometrically
isomorphic to the predual of a von Neumann algebra $\mathcal{M}_\mathcal{U}$. 
Further, it is shown in \cite{R} that $\mathcal{M}_\mathcal{U}$ has such a nice behaviour as
$\prod_\mathcal{U}L^p(\mathcal{M}_n)$ is isometrically isomorphic to $L^p(\mathcal{M}_\mathcal{U})$
for each $p<\infty$.
Specifically, 
\begin{itemize}
	\item 
	there exists an isometric $\ast$-homomorphism 
	\[
	\iota\colon {\textstyle\prod}_\mathcal{U}\mathcal{M}_n\to \mathcal{M}_\mathcal{U}
	\]
	from the $C^*$-algebra 
	$\prod_\mathcal{U}\mathcal{M}_n$ into the von Neumann algebra $\mathcal{M}_\mathcal{U}$
	such that $\iota\bigl(\prod_\mathcal{U}\mathcal{M}_n\bigr)$ is weak* dense in $\mathcal{M}_\mathcal{U}$, and,
	\item
	for each $p<\infty$, there exists an isometric isomorphism 
	\[
	\Lambda_p\colon{\textstyle\prod}_\mathcal{U} L^p(\mathcal{M}_n)\to L^p(\mathcal{M}_\mathcal{U})
	\]
	such that
	\begin{equation*}
	\Lambda_p\bigl(
	(a_n)_\mathcal{U}(x_n)_\mathcal{U}(b_n)_\mathcal{U}
	\bigr)=
	\iota\bigl((a_n)_\mathcal{U}\bigr)
	\Lambda_p\bigl((x_n)_\mathcal{U}\bigr)
	\iota\bigl((b_n)_\mathcal{U}\bigr)
	\end{equation*}
	and, for $0<p,q,r<\infty$ with $\tfrac{1}{p}+\tfrac{1}{q}=\tfrac{1}{r}$,
	\begin{equation*}
	\Lambda_r\bigl((x_n)_\mathcal{U}(y_n)_\mathcal{U}\bigr)=
	\Lambda_p\bigl((x_n)_\mathcal{U}\bigr)
	\Lambda_q\bigl((y_n)_\mathcal{U}\bigr)
	\end{equation*}
	for all $(a_n)_\mathcal{U},(b_n)_\mathcal{U}\in\prod_\mathcal{U}\mathcal{M}_n$,
	$(x_n)_\mathcal{U}\in\prod_\mathcal{U} L^p(\mathcal{M}_n)$, and
	$(y_n)_\mathcal{U}\in\prod_\mathcal{U} L^q(\mathcal{M}_n)$.
\end{itemize}
Actually, \cite{R} is concerned with the ultrapower of $L^p(\mathcal{M})$ for a given von Neumann algebra,
but it is also emphasized there that the results are equally valid for the above situation.

\begin{theorem}\label{t2018}
Let $1\le q<p\le\infty$.
Then there exists a constant $C_{p,q}\in\mathbb{R}^+$
with the property that,
for each von Neumann algebra $\mathcal{M}$ and
each continuous linear map $T\colon L^p(\mathcal{M})\to L^q(\mathcal{M})$, 
we have
\[
\dist 
\bigl(T,\Hom_\mathcal{M}(L^p(\mathcal{M}),L^q(\mathcal{M}))\bigr)\le
C_{p,q}
\sup\bigl\{\Vert e^\perp Te\Vert : e\in\mathcal{M}\text{ projection}\bigr\}.
\]
In particular, the space $\Hom_\mathcal{M}(L^p(\mathcal{M}),L^q(\mathcal{M}))$ is hyperreflexive.
\end{theorem}

\begin{proof}
In the case where either $p=\infty$ or $q=1$, we apply Theorem~\ref{ac1236}
to obtain the result.
	
Suppose that $1<q<p<\infty$, and take $1<r<\infty$ such that $\tfrac{1}{p}+\tfrac{1}{r}=\tfrac{1}{q}$.
Our objective is to prove that there exists a constant $c_{p,q}\in\mathbb{R}^+$ with the property that
for each von Neumann algebra $\mathcal{M}$ and
each $T\in B(L^p(\mathcal{M}),L^q(\mathcal{M}))$, 
we have
\begin{equation}\label{h1}
\dist 
\bigl(T,\Hom_\mathcal{M}\bigl(L^p(\mathcal{M}),L^q(\mathcal{M})\bigr)\bigr)\le
c_{p,q}\,\Vert\ad (T)\Vert.
\end{equation}
Assume towards a contradiction that the clause is false, and there is no such constant $c_{p,q}$.
Then, for each $n\in\mathbb{N}$, there exists a von Neumann algebra $\mathcal{M}_n$ and a continuous linear map
$R_n\colon L^p(\mathcal{M}_n)\to L^q(\mathcal{M}_n)$ such that
\begin{equation*}
\delta_n:=\dist \bigl(R_n,\Hom_{\mathcal{M}_n}\bigl(L^p(\mathcal{M}_n),L^q(\mathcal{M}_n)\bigr)\bigr)
>
n\Vert\ad (R_n)\Vert.
\end{equation*}
For each $n\in\mathbb{N}$, set $S_n=\delta_n^{-1}R_n$.
Then
\begin{equation}\label{1756}
\Vert\ad (S_n)\Vert< 1/n 
\quad 
(n\in\mathbb{N})
\end{equation}
and
\begin{equation}\label{1757}
\dist  \bigl(S_n,\Hom_{\mathcal{M}_n}\bigl(L^p(\mathcal{M}_n),L^q(\mathcal{M}_n)\bigr)\bigr) =1 
\quad 
(n\in\mathbb{N}).
\end{equation}
Since the sequence $( S_{n})$ need not to be bounded, 
we replace it with a bounded one that still satisfies both \eqref{1756} and \eqref{1757}. 
For this purpose, for each $n\in\mathbb{N}$, we take 
$\Psi_n\in\Hom_{\mathcal{M}_n}
\bigl(L^p(\mathcal{M}_n),L^q(\mathcal{M}_n)\bigr)$
such that $\Vert S_n-\Psi_n\Vert<1+1/n$ and consider the map $T_n=S_n-\Psi_n$. 
Then $(T_n)$ is bounded.
It is clear that $\Vert\ad (T_n)\Vert=\Vert\ad (S_n)\Vert$ and that
\[
\dist  \bigl(T_n,\Hom_{\mathcal{M}_n}\bigl(L^p(\mathcal{M}_n),L^q(\mathcal{M}_n)\bigr)\bigr)=
\dist  \bigl(S_n,\Hom_{\mathcal{M}_n}\bigl(L^p(\mathcal{M}_n),L^q(\mathcal{M}_n)\bigr)\bigr)
\] 
for each $n\in\mathbb{N}$, so that \eqref{1756} and \eqref{1757} give
\begin{equation}\label{1759}
\Vert\ad  (T_n)\Vert< 1/n 
\quad 
(n\in\mathbb{N}),
\end{equation}
\begin{equation}\label{1760}
\dist  \bigl(T_n,\Hom_{\mathcal{M}_n}\bigl(L^p(\mathcal{M}_n),L^q(\mathcal{M}_n)\bigr)\bigr)=1 
\quad 
(n\in\mathbb{N}).
\end{equation}
	
Take a free ultrafilter $\mathcal{U}$ on $\mathbb{N}$.
Consider the ultraproduct von Neumann algebra 
\[
\mathcal{M}_\mathcal{U}=
\left({\textstyle\prod}_\mathcal{U}L^1(\mathcal{M}_n)\right)^*
\]
and the maps
\begin{gather*}
\iota\colon {\textstyle\prod}_\mathcal{U}\mathcal{M}_n\to\mathcal{M}_\mathcal{U},
\\
\Lambda_p\colon{\textstyle\prod}_\mathcal{U} L^p(\mathcal{M}_n)\to L^p(\mathcal{M}_\mathcal{U}),
\\
\Lambda_q\colon{\textstyle\prod}_\mathcal{U} L^q(\mathcal{M}_n)\to L^q(\mathcal{M}_\mathcal{U}),
\\
\Lambda_r\colon{\textstyle\prod}_\mathcal{U} L^r(\mathcal{M}_n)\to L^r(\mathcal{M}_\mathcal{U})
\end{gather*}
introduced in the preliminary remark.
Further, take the ultraproduct map
\[
{\textstyle\prod}_\mathcal{U}T_n\colon 
{\textstyle\prod}_\mathcal{U}L^p(\mathcal{M}_n)\to{\textstyle\prod}_\mathcal{U} L^q(\mathcal{M}_n).
\]
We claim that $\prod_\mathcal{U}T_n$ is a right ${\textstyle\prod}_\mathcal{U}\mathcal{M}_n$-module homomorphism.
Take elements
$(x_n)_\mathcal{U}\in {\textstyle\prod}_\mathcal{U}L^p(\mathcal{M}_n)$ and
$(a_n)_\mathcal{U}\in {\textstyle\prod}_\mathcal{U}\mathcal{M}_n$. 
Then \eqref{1759} gives
\begin{equation*}
\begin{split}
\left\Vert
{\textstyle\prod}_\mathcal{U}T_n
\bigl((x_n)_\mathcal{U}(a_n)_\mathcal{U}\bigr)
-
{\textstyle\prod}_\mathcal{U}T_n
\bigl((x_n)_\mathcal{U}\bigr)(a_n)_\mathcal{U}
\right\Vert & =
\lim_\mathcal{U} \left\Vert {T}_n(x_na_n)-T_n(x_n)a_n\right\Vert \\
& \le
\lim_\mathcal{U}\bigl(\Vert\ad (T_n)\Vert \Vert x_n\Vert\Vert a_n\Vert\bigr) \\
& \le
\lim_\mathcal{U}\bigl(\tfrac{1}{n}\Vert x_n\Vert\Vert a_n\Vert\bigr)
=0.
\end{split}
\end{equation*}
Define 
$\mathbf{T}\colon L^p(\mathcal{M}_\mathcal{U})\to L^q(\mathcal{M}_\mathcal{U})$ by
\[
\mathbf{T}=\Lambda_q\circ{\textstyle\prod}_\mathcal{U}T_n\circ{\Lambda_p}^{-1}.
\]
Then $\mathbf{T}$ is a  right $\iota\bigl({\textstyle\prod}_\mathcal{U}\mathcal{M}_n\bigr)$-module homomorphism.
We now note that: 
\begin{itemize}
\item
$\iota\bigl({\textstyle\prod}_\mathcal{U}\mathcal{M}_n\bigr)$ is weak* dense in $\mathcal{M}_\mathcal{U}$;
\item
the module maps $\mathbf{a}\mapsto\mathbf{x}\mathbf{a}$ and 
$\mathbf{a}\mapsto\mathbf{y}\mathbf{a}$ are weak*-weak* continuous for all 
$\mathbf{x}\in L^p(\mathcal{M}_\mathcal{U})$ and $\mathbf{y}\in L^q(\mathcal{M}_\mathcal{U})$;
\item
the map $\mathbf{T}$ is weak*-weak* continuous,
since both $L^p(\mathcal{M}_\mathcal{U})$ and $L^q(\mathcal{M}_\mathcal{U})$
are reflexive (being  $1<p,q<\infty$).
\end{itemize}
The  above conditions imply that
$\mathbf{T}$ is a right $\mathcal{M}_\mathcal{U}$-module homomorphism.
By Theorem~\ref{js}, there exists $\Xi\in L^r(\mathcal{M}_\mathcal{U})$ such that
\[
\mathbf{T}(\mathbf{x})=\Xi\mathbf{x}
\quad
(\mathbf{x}\in L^p(\mathcal{M}_\mathcal{U})).
\]
Set
$(\xi_n)_\mathcal{U}={\Lambda_r}^{-1}(\Xi)\in {\textstyle\prod}_\mathcal{U}L^r(\mathcal{M}_n)$,
and, for each $n\in\mathbb{N}$, take the left composition map $L_{\xi_n}\colon L^p(\mathcal{M}_n)\to L^q(\mathcal{M}_n)$.
Then, for each $\mathbf{x}\in L^p(\mathcal{M}_\mathcal{U})$, we have
\begin{equation*}
\begin{split}
\mathbf{T}(\mathbf{x}) & =
\Xi\mathbf{x}=
\Lambda_r\bigl((\xi_n)_\mathcal{U}\bigr)
\Lambda_p\bigl({\Lambda_p}^{-1}(\mathbf{x})\bigr)  =
\Lambda_q\bigl((\xi_n)_\mathcal{U}{\Lambda_p}^{-1}(\mathbf{x})\bigr) \\
& = \bigl(\Lambda_q\circ {\textstyle\prod}_\mathcal{U}L_{\xi_n}\circ{\Lambda_p}^{-1}\bigr)(\mathbf{x}),
\end{split}
\end{equation*}
whence 
${\textstyle\prod}_\mathcal{U}T_n={\textstyle\prod}_\mathcal{U}L_{\xi_n}$,
so that \eqref{e2027} gives
$\lim_\mathcal{U}\Vert T_n-L_{\xi_n}\Vert=0$
and hence
\[
\lim_\mathcal{U}
\dist \bigl(T_n,
\Hom_{\mathcal{M}_n}\bigl(L^p(\mathcal{M}_n),L^q(\mathcal{M}_n)\bigr)\bigr)
\le\lim_\mathcal{U}\Vert T_n-L_{\xi_n}\Vert=0,
\]
contrary to \eqref{1760}.

Finally, \eqref{h1} and Corollary~\ref{c1549} give the desired inequality with $C_{p,q}=8c_{p,q}$.

The hyperreflexivity follows from the estimates in Proposition~\ref{p10}.
\end{proof}

\end{document}